\documentclass[11pt]{article}
\usepackage{amssymb}
\usepackage{amsmath}
\usepackage{amsfonts}

\usepackage[body=16cm,top=4cm,bottom=4cm]{geometry}
\usepackage[english]{babel}

\newtheorem{theorem}{Theorem}

\newtheorem{definition}[theorem]{Definition}

\newtheorem{lemma}[theorem]{Lemma}

\newtheorem{proposition}[theorem]{Proposition}
\newtheorem{remark}[theorem]{Remark}

\newenvironment{proof}[1][Proof]{\noindent\textbf{#1.} }{\ \rule{0.5em}{0.5em}}
\input{tcilatex}

\begin{document}

\begin{center}
{\Large A Bismut-Elworthy-Li Formula for Singular SDE's Driven by a
Fractional Brownian Motion and Applications to Rough Volatility Modeling%
}

\bigskip

Oussama Amine{  \footnote{{  email: oussamaa@math.uio.no}}}$^{%
\text{,6}}$, Emmanuel Coffie{  \footnote{{  email:
emmanuco@student.matnat.uio.no}}}$^{\text{,}}${  \footnote{{  %
African Institute for Mathematical Sciences (AIMS), Ghana, Biriwa, N1,\
Accra-Cape Coast Road, Ghana.}}}, Fabian Harang{  \footnote{{  %
email: fabianah@math.uio.no}}}$^{\text{,6}}$ and Frank Proske{  
\footnote{{  email: proske@math.uio.no}}}$^{\text{,}}${  \footnote{%
{  Department of Mathematics, University of Oslo, Moltke Moes vei 35,
P.O. Box 1053 Blindern, 0316 Oslo, Norway.}}}

{\Large \bigskip }

{\large Abstract}
\end{center}

\bigskip In this paper we derive a Bismut-Elworthy-Li type formula with
respect to strong solutions to singular stochastic differential equations
(SDE's) with additive noise given by a multi-dimensional fractional Brownian
motion with Hurst parameter $H<1/2$. "Singular" here means that the drift
vector field of such equations is allowed to be merely bounded and
integrable. As an application we use this representation formula for the
study of price sensitivities of financial claims based on a stock price
model with stochastic volatility, whose dynamics is described by means of
fractional Brownian motion driven SDE's.

Our approach for obtaining these results is based on Malliavin calculus and
arguments of a recently developed "local time variational calculus".

\bigskip \emph{keywords}: Bismut-Elworthy-Li formula, singular SDEs,
fractional Brownian motion, Malliavin calculus, stochastic flows, stochastic
volatility

\emph{Mathematics Subject Classification} (2010): 60H10, 49N60, 91G80.

\bigskip

\section{Introduction}

In recent years the construction and computation of risk measures have
become an indispensable tool for the risk analysis and risk management of
portfolios in banks and insurance companies worldwide. An important class of risk
measures often applied by investors on financial markets to hedge their
positions is given by the \textquotedblright greeks\textquotedblright . These are market sensitivities usually denoted by Greek letters e.g. \textquotedblright Delta\textquotedblright , \textquotedblright
Gamma\textquotedblright , \textquotedblright Rho\textquotedblright ,
\textquotedblright Theta\textquotedblright , "Vega"..., and hence the name. For example the
Delta $\Delta $, which can be used for the construction of delta hedges in
portfolio management, measures the sensitivity of price changes of financial
derivatives with respect to the initial price of the underlying asset. Roughly
speaking, greeks \ are derivatives with respect to a parameter $\lambda $ of
a (risk-neutral) price, that is, for example of the form 
\begin{equation}
\frac{\partial }{\partial \lambda }E[\Phi ((X_{T}^{\lambda }))],  \label{1}
\end{equation}%
where $\Phi $ is the payoff function of a claim and $X_{T}^{\lambda }$ the
underlying asset at terminal time $T$, which depends on $\lambda $.

In general, greeks cannot be obtained by closed-form formulas, especially in the
case of discontinuous payoff functions. Therefore, one has to resort
to numerical techniques to approximate such sensitivities. A ground breaking
method in this direction, which is also applicable to path-dependent
options, has been developed in Fourni\'{e} et al. \cite{F1}, \cite{F2}.
Assuming that the dynamics of asset prices $X_{t}=X_{t}^{\lambda }$ is
modeled by a stochastic differential equation of the form 
\begin{equation}
dX(t)=b(t,X(t))dt+\sigma (t,X(t))dW_{t},X_{0}=x\in \mathbb{R}^{d},0\leq
t\leq T,  \label{2}
\end{equation}%
where $W_{t},0\leq t\leq T$ is a $d-$dimensional Wiener process and $b,$ $%
\sigma $ are continuously differentiable coefficients, the authors in \cite%
{F1} were able to represent (\ref{1}) in a derivative-free form,
that is by 
\begin{equation}
E[\phi (S(T))\pi ],  \label{3}
\end{equation}%
where $\pi $ is the so called Malliavin weight. Such a representation is
also referred to as Bismut-Elworthy-Li formula (BEL-formula) in the
literature.\ See \cite{BEL1} and \cite{BEL2}.

An advantage of this method is that the representation in (\ref{3}) does not
involve derivatives of $\Phi $ and that it exhibits numerical tractability
via efficient use of Monte-Carlo simulation. However, a deficiency of this
approach is the requirement that the coefficients of the SDE, which
describes the dynamics of the asset prices in (\ref{2}), are continuously
differentiable. The latter assumption is rather restrictive and excludes the
study of interesting financial models.\ Such models could e.g. pertain to a
generalization of the Black-Scholes model with "regime-switching" drift,
that is 
\begin{equation}
S_{t}^{x}=x\exp (Y_{t}),  \label{4}
\end{equation}%
where%
\begin{equation*}
dY_{t}=(b_{1}\chi _{\{Y_{t}>R\}}+b_{2}\chi _{\{Y_{t}\leq R\}})dt-\frac{1}{2}%
\sigma ^{2}dt+\sigma dB_{t}
\end{equation*}%
for constants $b_{1}$, $b_{2}$ and a "threshold" $R.$

Another possible application is to interest rate or commodity markets with a model whose dynamics is given by a generalized
Ornstein-Uhlenbeck process with regime switching mean reversion rate, that
is 
\begin{equation}
dY_{t}=(a_{1}\chi _{\{Y_{t}>R\}}+a_{2}\chi _{\{Y_{t}\leq
R\}})(b-Y_{t})dt+\sigma dW_{t}  \label{5}
\end{equation}%
for mean reversion coefficients $a_{1},a_{2}>0$, a threshold $R$, the
long-run average level $b\in \mathbb{R}$, interest rate volatility $\sigma
>0 $.

In the above models (\ref{4}) and (\ref{5}) the drift coefficients are
chosen to be discontinuous and used to capture regime-switching effects which may arise from regulations, credit rating changes, market crashes or other financial
disasters.

We mention that a BEL-representation for Wiener process driven SDE's with
merely bounded and measurable drift functions as e.g. the one in (\ref{5}) was
first obtained in Menoukeu-Pamen et al. \cite[Theorem 4.6, Remark 4.7]{MMNPZ}%
. To be more precise, for strong solutions $X_{t},0\leq t\leq T$ to SDE's
with additive Wiener noise 
\begin{equation*}
dX_{t}^{x}=b(t,X_{t}^{x})dt+dW_{t},
\end{equation*}%
where $b\in L^{\infty }([0,T]\times \mathbb{R}^{d};\mathbb{R}^{d})$, the
authors prove, for bounded Borel measurable $\Phi $ and bounded open sets $%
U\subset \mathbb{R}^{d}$, that%
\begin{equation}
\frac{\partial }{\partial x}E[\Phi (X_{T}^{x})]=E[\Phi (X_{T}^{x})\pi
]^{\ast }  \label{W}
\end{equation}%
for all $x\in U$ a.e., where the Malliavin weight $\pi $ is 
\begin{equation*}
\int_{0}^{T}a(s)\left( \frac{\partial }{\partial x}X_{s}^{x}\right) ^{\ast
}dW_{s}.
\end{equation*}%
Here the derivatives appearing on both sides of (\ref{W}) are Sobolev
derivatives on $U$, $a:[0,T]\longrightarrow \mathbb{R}$ is a bounded Borel
measurable function with $\int_{0}^{T}a(s)ds=1$ and $\ast $ denotes
transposition. See also the related articles \cite{BMDP}, \cite{BDMP}, \cite%
{XZhang} and the references therein.

Using techniques from Malliavin calculus and arguments of a "local time
variational calculus" as recently developed in the series of works \cite{BNP}%
, \cite{BLPP}, \cite{ABP} in the case of fractional Brownian motion, we aim
at obtaining in this paper an extension of the above mentioned results to
the case of fractional Brownian motion driven singular SDE's. More
precisely, we want to derive a BEL-formula of the type (\ref{W}) with
respect to strong solutions to SDE's of the form

\begin{equation}
dX_{t}=b(t,X_{t})dt+dB_{t}^{H},X_{0}=x,0\leq t\leq T,  \label{fB}
\end{equation}%
where $B_{t}^{H},0\leq t\leq T$ is a $d-$dimensional fractional Brownian
motion with Hurst parameter $H\in (0,\frac{1}{2})$ and where the vector
field $b$ is \emph{singular} in the sense that%
\begin{equation*}
b\in L_{\infty ,\infty }^{1,\infty }:=L^{1}(\mathbb{R}^{d};L^{\infty }([0,T];%
\mathbb{R}^{d}))\cap L^{\infty }(\mathbb{R}^{d};L^{\infty }([0,T];\mathbb{R}%
^{d})).
\end{equation*}%
\bigskip

As an application of the techniques used in connection with the BEL-formula,
we also wish to study a Black-Scholes model with "turbulent" stochastic
volatility, where the dynamics of stock prices is described by the
(singular) SDE%
\begin{equation*}
dX_{t}^{x}=\mu X_{t}^{x}dt+\sigma _{t}X_{t}^{x}dB_{t},X_{0}^{x}=x,0\leq
t\leq T.
\end{equation*}%
Here $B_{t},0\leq t\leq T$ is a one-dimensional Wiener process, $\mu $ the
mean return and $\sigma _{t}$ the volatility at time $t$, modeled by means
of the SDE%
\begin{equation*}
dY_{t}^{y}=b(t,Y_{t}^{y})dt+B_{t}^{H},Y_{0}^{y}=y,0\leq t\leq T
\end{equation*}%
for small Hurst parameters $H\in (0,1/2)$ and singular vector fields $b\in
L_{\infty ,\infty }^{1,\infty }$, which can be used as explained above for
the modeling of regime switching effects in stock markets. Let us also
mention that the choice of fractional Brownian motion with small Hurst
parameters $H$ in the latter model, which becomes "rougher" the lower $H$
is, is in fact supported by empirical evidence (see \cite{G}) and useful for
the description of stock price volatilities $\sigma _{t}$ in "turbulent"
stock markets.

\bigskip

Finally, we also point out the interesting work \cite{FanRen}, where the authors derived
BEL-formulas for (functional) SDE's driven by fractional Brownian motion
with Hurst parameters $H\in (0,1)$ in the case of differentiable vector
fields, which they applied to e.g. the study of Harnack type of inequalities.

\bigskip

The paper is organized as follows: In Section 2 we prove a BEL-formula with
respect to the SDE (\ref{fB}) for $H<\frac{1}{2(d+2)}$. See Theorem \ref%
{Bismut}. We then show, in Proposition \ref{ContinuousBEL}, that the
BEL-representation has a continuous version, if $H<\frac{1}{2(d+3)}$.
Finally, in\ Section 3 we discuss an application of our techniques used in
Section 2 to the sensitivity analysis of prices of options based on a
Black-Scholes model with "rough" stochastic volatility (Theorem \ref%
{Volatility}).

\section{Bismut-Elworthy-Li formula}

In this section we aim at deriving a new Bismut-Elworthy-Li type formula
with respect to SDE's driven by discontinuous vector fields and a fractional
Brownian motion with a Hurst parameter $H<\frac{1}{2}$. We also propose a
stock price model with "rough" stochastic volatility, which allows for the
description of regime switching effects with respect to volatility data
caused e.g. by economical crises, political changes or other shocks on
markets. Here, regime switching effects are modeled by means of singular
coefficients of SDE's driven by a fractional Brownian motion. On the other
hand, the "roughness" of the volatility paths in the sense of paths with low
H\"{o}lder regularity is described through the driving fractional noise of
such SDE's. Further, we also prove a BEL-representation for the delta of an
option with respect to that model.

In what follows, let us consider a fractional Brownian motion $%
B_{t}^{H},t\geq 0$ with Hurst parameter $H\in (0,1)$ on some complete
probability space $(\Omega ,\mathcal{F},\mu )$, which is (in the $1-$%
dimensional case) a centered Gaussian process with a covariance structure $%
R_{H}(t,s)$ of the form%
\begin{equation*}
R_{H}(t,s)=E[B_{t}^{H}B_{s}^{H}]=\frac{1}{2}(s^{2H}+t^{2H}-\left\vert
t-s\right\vert ^{2H})
\end{equation*}%
for all $t,s\geq 0$. See the Appendix. In the special case, when $H=\frac{1}{%
2}$ the fractional Brownian motion coincides with a Wiener process.

We also recall that the fractional Brownian motion is self-similar, that is%
\begin{equation*}
\{B_{\alpha t}^{H}\}_{t\geq 0}\overset{law}{=}\{\alpha
^{H}B_{t}^{H}\}_{t\geq 0}
\end{equation*}%
for all $\alpha >0$. Further, $B^{H}$ has a version with paths, which are $%
(H-\varepsilon )$-H\"{o}lder continuous for all $\varepsilon \in (0,H)$.
Another propertie satisfied by $B^{H}$, which actually rather complicate
the study of fractional Brownian motion, is that it is neither a Markov
process nor a semimartingale, when $H\neq \frac{1}{2}$. See e.g. \cite{Nualart}
and the references therein for further information on the fractional
Brownian motion.

In this Section, we consider for $H<\frac{1}{2}$ the SDE

\begin{equation}
dX_{t}^{x}=b(t,X_{t}^{x})dt+dB_{t}^{H},X_{0}^{x}=x,0\leq t\leq T.
\label{SDE}
\end{equation}

We mention that $B^{H}$ in this case has the representation 
\begin{equation}
B_{t}^{H}=\int_{0}^{t}K_{H}(t,s)I_{d\times d}dB_{s}  \label{RepFractional}
\end{equation}%
for a $d-$dimensional Brownian motion $B_{\cdot }$, where $I_{d\times d}\in 
\mathbb{R}^{d\times d}$ is the unit matrix and $K_{H}$ the kernel as given
in (\ref{KH}) in the Appendix.

In the sequel, we also need the following notation for function spaces:%
\begin{eqnarray*}
L_{\infty }^{1} &:&=L^{1}(\mathbb{R}^{d};L^{\infty }([0,T];\mathbb{R}^{d})),
\\
L_{\infty }^{\infty } &:&=L^{\infty }(\mathbb{R}^{d};L^{\infty }([0,T];%
\mathbb{R}^{d})), \\
L_{\infty ,\infty }^{1,\infty } &:&=L_{\infty }^{1}\cap L_{\infty }^{\infty
}.
\end{eqnarray*}

We have the following result for the existence and uniqueness of strong
solutions to the SDE (\ref{SDE}) which is due to \cite{BNP} (compare also
the results in \cite{CG} and \cite{NO}, which cannot be used to treat the
case $b\in L_{\infty ,\infty }^{1,\infty }$ for $d>1$):

\begin{theorem}
\label{StrongSolution}Let $b\in L_{\infty ,\infty }^{1,\infty }$. Then if $H<%
\frac{1}{2(d+2)}$ there exists a unique (global) strong solution $X_{\cdot
}^{x}$ of the SDE (\ref{SDE}). Moreover, for every $x\in \mathbb{R}^{d},t\in
\lbrack 0,T]$ $X_{t}^{x}$ is Malliavin differentiable in the direction of
the Brownian motion $B$ in (\ref{RepFractional}) and $X_{t}^{\cdot }$ is
locally Sobolev differentiable $\mu -a.e.$ \par
That is, more precisely, $$%
X_{t}^{\cdot }\in \dbigcap\limits_{p\geq 2}L^{2}(\Omega ;W^{1,p}(U))$$ for
bounded and open sets $U\subset \mathbb{R}^{d}$.
\end{theorem}

\bigskip

In preparation of our main result (Theorem \ref{Bismut}), we also need a
series of auxiliary results:

\begin{lemma}
\bigskip \label{BoundedDerivatives}Let $b\in C_{c}^{\infty }((0,T)\times 
\mathbb{R}^{d}).$ Fix integers $p\geq 2$. Then, if $H<\frac{1}{2(d+2)},$ we
have 
\begin{equation*}
\sup_{x\in \mathbb{R}^{d}}E[\left\Vert \frac{\partial }{\partial x}%
X_{t}^{x}\right\Vert ^{p}]\leq C_{p,H,d,T}(\left\Vert b\right\Vert
_{L_{\infty }^{\infty }},\left\Vert b\right\Vert _{L_{\infty }^{1}})<\infty
\end{equation*}%
for some continuous function $C_{p,H,d,T}:[0,\infty )^{2}\longrightarrow
\lbrack 0,\infty )$.
\end{lemma}

\begin{proof}
See \cite{BNP}.
\end{proof}

\bigskip

\begin{lemma}
\label{BoundMalliavinDerivative}\bigskip Let $H<\frac{1}{2(d+2)}$, $b\in
L_{\infty ,\infty }^{1,\infty }$. Further, let $X_{\cdot }^{n},n\geq 1$ be
the sequence of strong solutions to (\ref{SDE}) associated with functions $%
b_{n}\in C_{c}^{\infty }((0,T)\times \mathbb{R}^{d}),n\geq 1$ such that%
\begin{equation}
b_{n}(t,x)\underset{n\longrightarrow \infty }{\longrightarrow }b(t,x)\text{ }%
(t,x)-a.e.,  \label{C1}
\end{equation}%
\begin{equation}
\sup_{n\geq 1}\left\Vert b_{n}\right\Vert _{L_{\infty }^{1}}<\infty
\label{C2}
\end{equation}%
and 
\begin{equation}
\left\vert b(t,x)\right\vert \leq M<\infty ,n\geq 1\text{ a.e. for some
constant }M.  \label{C3}
\end{equation}%
Fix $t\in \lbrack 0,T]$ and $x\in \mathbb{R}^{d}.$ Then there exists a $%
\beta \in (0,1/2)$ such that%
\begin{equation*}
\sup_{n\geq 1}\int_{0}^{t}\int_{0}^{t}\frac{E[\left\Vert D_{\theta
}X_{t}^{n}-D_{\theta ^{\prime }}X_{t}^{n}\right\Vert ^{2}]}{\left\vert
\theta -\theta ^{\prime }\right\vert ^{1+2\beta }}d\theta ^{\prime }d\theta
\leq \sup_{n\geq 1}C_{H,d,T}(\left\Vert b_{n}\right\Vert _{L_{\infty
}^{\infty }},\left\Vert b_{n}\right\Vert _{L_{\infty }^{1}})<\infty
\end{equation*}%
and%
\begin{equation}
\sup_{n\geq 1}\left\Vert D_{\cdot }X_{t}^{n}\right\Vert _{L^{2}(\Omega
\times \lbrack 0,T])}\leq \sup_{n\geq 1}C_{H,d,T}(\left\Vert
b_{n}\right\Vert _{L_{\infty }^{\infty }},\left\Vert b_{n}\right\Vert
_{L_{\infty }^{1}})<\infty  \label{D}
\end{equation}%
for some continuous function $C_{H,d,T}:[0,\infty )^{2}\longrightarrow
\lbrack 0,\infty ).$
\end{lemma}

\begin{proof}
See \cite{BNP}.
\end{proof}

\bigskip

\begin{proposition}
\label{ConvergenceX}Let $X_{\cdot }^{x,n},n\geq 1$ be a sequence of strong
solutions as in Lemma \ref{BoundMalliavinDerivative} and $X_{\cdot }$ the
strong solution to (\ref{SDE}). Then%
\begin{equation*}
X_{t}^{x,n}\underset{n\longrightarrow \infty }{\longrightarrow }X_{t}^{x}%
\text{ in }L^{2}(\Omega )
\end{equation*}%
for all $t,x.$
\end{proposition}

\begin{proof}
See \cite{BNP}.
\end{proof}

\bigskip

\begin{lemma}
\label{WeakConvergence}\bigskip\ Let $U\subset \mathbb{R}^{d}$ be an open
and bounded subset. Consider the sequence $X_{\cdot }^{x,n},n\geq 1$ in
Proposition \ref{ConvergenceX}. Then%
\begin{equation*}
\frac{\partial }{\partial x}X_{\cdot }^{\cdot ,n}\underset{n\longrightarrow
\infty }{\longrightarrow }\frac{\partial }{\partial x}X_{\cdot }^{\cdot }
\end{equation*}%
in $L^{2}([0,T]\times \Omega \times U)$ weakly.

\begin{proof}
This result is a consequence of Proposition \ref{ConvergenceX} and the
estimate in Lemma \ref{BoundedDerivatives}.
\end{proof}
\end{lemma}

\bigskip

We are coming now to the main result of our article:

\begin{theorem}[Bismut-Elworthy-Li formula]
\label{Bismut}Let $H<\frac{1}{2(d+2)}$ and let $X_{\cdot }^{x}$ be the
unique strong solution to the SDE%
\begin{equation*}
dX_{t}^{x}=b(t,X_{t}^{x})dt+dB_{t}^{H},X_{0}^{x}=x,0\leq t\leq T
\end{equation*}%
for $b\in L_{\infty ,\infty }^{1,\infty }$. Further, assume that $U$ is a
bounded and open subset of $\mathbb{R}^{d}$ and $\Phi :\mathbb{R}%
^{d}\longrightarrow \mathbb{R}$ a Borel measurable function such that%
\begin{equation*}
\Phi (X_{T}^{\cdot })\in L^{2}(\Omega \times U,\mu \times dx).
\end{equation*}%
. In addition, consider a bounded Borel measurable function $%
a:[0,T]\longrightarrow \mathbb{R}$ such that\ 
\begin{equation*}
\int_{0}^{T}a(s)ds=1.
\end{equation*}%
Then%
\begin{eqnarray}
&&\frac{\partial }{\partial x}E[\Phi (X_{T}^{x})]  \notag \\
&=&C_{H}E[\Phi (X_{T}^{x})\int_{0}^{T}u^{-H-\frac{1}{2}%
}\int_{u}^{T}a(s-u)(s-u)^{\frac{1}{2}-H}s^{H-\frac{1}{2}}\left( \frac{%
\partial }{\partial x}X_{s-u}^{x}\right) ^{\ast }dB_{s}du]^{\ast }
\label{BEL}
\end{eqnarray}%
for all $x\in U$ a.e., $0<t\leq T$, where $\ast $ denotes the transposition
of matrices and where $C_{H}=1/(c_{H}\Gamma (\frac{1}{2}+H)\Gamma (\frac{1}{2%
}-H))$ for%
\begin{equation*}
c_{H}=(\frac{2H}{(1-2H)B(1-2H,H+1/2)})^{1/2}.
\end{equation*}%
Here $\Gamma $ and $B$ are the Gamma and Beta function, respectively.
\end{theorem}

\begin{remark}
Let $\mathcal{P}$ be the predictable $\sigma -$algebra with respect to the $%
\mu -$augmented filtration $\{\mathcal{F}_{t}\}_{0\leq t\leq T}$ generated
by $B_{\cdot }^{H}$. Then $\frac{\partial }{\partial x}X_{t}^{x},0\leq t\leq
T$ on the right hand side of Theorem \ref{Bismut} stands for a process $%
Y:[0,T]\times \Omega \times U\longrightarrow \mathbb{R}^{d\times d}$ in $%
L^{2}([0,T]\times \Omega \times U,\mathcal{P}\otimes \mathcal{B}(U);\mathbb{R%
}^{d\times d})$ such that $Y_{t}^{\cdot }(\omega )$ is the Sobolev
derivative of $X_{t}^{\cdot }(\omega )$ $(t,\omega )-$a.e.
\end{remark}

\begin{proof}
Let $\Phi \in C_{c}^{\infty }(\mathbb{R}^{d})$ and choose a sequence of
functions $b_{n}\in C_{c}^{\infty }((0,T)\times \mathbb{R}^{d})$, which
approximates the vector field $b$ in the sense of (\ref{C1}), (\ref{C2}) and
(\ref{C3}).

Denote by $X_{\cdot }^{s,x,n}$ the unique strong solution to 
\begin{equation*}
dX_{t}^{s,x,n}=b_{n}(t,X_{t}^{s,x,n})dt+dB_{t}^{H},X_{s}^{s,x,n}=x,s\leq
t\leq T
\end{equation*}%
for all $n$. Since $b_{n}\in C_{c}^{\infty }((0,T)\times \mathbb{R}^{d})$,
it follows that there exists a $\Omega ^{\ast }$ with $\mu (\Omega ^{\ast
})=1$ such that for all $\omega \in \Omega ^{\ast },0\leq s\leq t\leq T$%
\begin{equation*}
(x\mapsto X_{t}^{s,x,n}(\omega ))\in C^{\infty }(\mathbb{R}^{d}).
\end{equation*}%
See e.g. \cite{Kunita}.

The latter and dominated convergence then give%
\begin{equation*}
\frac{\partial }{\partial x}E[\Phi (X_{T}^{x,n})]=E[\Phi ^{\shortmid
}(X_{T}^{x,n})\frac{\partial }{\partial x}X_{T}^{x,n}],
\end{equation*}%
where $\Phi ^{\shortmid }$ is the derivative of $\Phi $ and $%
X_{t}^{x,n}=X_{t}^{0,x,n}$. On the other hand, we have that for all $0\leq
s\leq t\leq T,x\in U$%
\begin{equation*}
X_{t}^{x,n}=X_{t}^{s,X_{s}^{x,n},n}\text{ a.e.}
\end{equation*}%
So we obtain that%
\begin{equation*}
\frac{\partial }{\partial x}E[\Phi (X_{T}^{x,n})]=E[\Phi ^{\shortmid
}(X_{T}^{x,n})\frac{\partial }{\partial x}X_{T}^{s,X_{s}^{x,n},n}\frac{%
\partial }{\partial x}X_{s}^{x,n}].
\end{equation*}%
We also know that the Malliavin derivative $D_{\cdot }^{H}X_{t}^{s,x,n}$ of $%
X_{t}^{s,x,n}$ in the direction of $B_{\cdot }^{H}$ exists and satisfies the
equation%
\begin{equation*}
D_{u}^{H}X_{t}^{s,x,n}=\int_{u}^{t}b_{n}^{\shortmid
}(t,X_{r}^{s,x,n})D_{u}^{H}X_{r}^{s,x,n}dr+\chi _{_{(s,t]}(u)}I_{d\times d},
\end{equation*}%
where $I_{d\times d}$ is the identity matrix. Further, we see that $\frac{%
\partial }{\partial x}X_{\cdot }^{u,X_{u}^{x,n},n}$ solves the same equation
for $s=0$. Therefore, we obtain by uniqueness of solutions that%
\begin{equation*}
D_{u}^{H}X_{t}^{x,n}=\frac{\partial }{\partial x}X_{t}^{u,X_{u}^{x,n},n}
\end{equation*}%
a.e. Hence%
\begin{equation*}
\frac{\partial }{\partial x}E[\Phi (X_{T}^{x,n})]=E[\Phi ^{\shortmid
}(X_{T}^{x,n})D_{s}^{H}X_{T}^{x,n}\frac{\partial }{\partial x}X_{s}^{x,n}].
\end{equation*}%
Let $\varphi \in C_{c}^{\infty }(U)$. Then%
\begin{equation*}
-\int_{U}E[\Phi (X_{T}^{x,n})]\frac{\partial }{\partial x}\varphi
(x)dx=\int_{U}\varphi (x)E[\Phi ^{\shortmid
}(X_{T}^{x,n})D_{s}^{H}X_{T}^{x,n}\frac{\partial }{\partial x}X_{s}^{x,n}]dx.
\end{equation*}%
Further, using the fact that the function $a$ sums up to one combined with
the chain rule for $D_{\cdot }^{H}$ (see \cite{Nualart}), we obtain that%
\begin{eqnarray*}
&&-\int_{U}E[\Phi (X_{T}^{x,n})]\frac{\partial }{\partial x}\varphi (x)dx \\
&=&\int_{U}\varphi (x)E[\int_{0}^{T}\{a(s)\Phi ^{\shortmid
}(X_{T}^{x,n})D_{s}^{H}X_{T}^{x,n}\frac{\partial }{\partial x}%
X_{s}^{x,n}\}ds]dx \\
&=&\int_{U}\varphi (x)E[\int_{0}^{T}\{a(s)D_{s}^{H}\Phi (X_{T}^{x,n})\frac{%
\partial }{\partial x}X_{s}^{x,n}\}ds]dx
\end{eqnarray*}

On the other hand, Proposition 5.2.1 and p. 285 in \cite{Nualart} shows that%
\begin{equation*}
D_{s}^{H}\Phi (X_{T}^{x,n})=Cs^{\frac{1}{2}-H}\Big(\int_{s}^{T}(u-s)^{-H-\frac{1%
}{2}}u^{H-\frac{1}{2}}D_{u}\Phi (X_{T}^{x,n})du\Big).
\end{equation*}%
for a constant $C$ depending on $H$. $D_.$ stands here for the Malliavin derivative in the direction of the Brownian motion $B_.$.\par Hence, we obtain by substitution (first
for $u$ substituted by $u+s$ in the above relation and then for $s$ by $s-u$
in the next step), Fubini's theorem and the duality formula with respect to
the Malliavin derivative $D_{\cdot }$ that%
\begin{eqnarray*}
&&-\int_{U}E[\Phi (X_{T}^{x,n})]\frac{\partial }{\partial x}\varphi (x)dx \\
&=&C\int_{U}\varphi (x)E[\int_{0}^{T}\{a(s)Cs^{\frac{1}{2}-H} \\
&&\times (\int_{s}^{T}(u-s)^{-H-\frac{1}{2}}u^{H-\frac{1}{2}}D_{u}\Phi
(X_{T}^{x,n})du)\frac{\partial }{\partial x}X_{s}^{x,n}\}ds]dx \\
&=&C\int_{U}\varphi (x)E[\int_{0}^{T}u^{-H-\frac{1}{2}} \\
&&\times \int_{u}^{T}a(s-u)(s-u)^{\frac{1}{2}-H}s^{H-\frac{1}{2}}D_{s}\Phi
(X_{T}^{x,n})\frac{\partial }{\partial x}X_{s-u}^{x,n}dsdu]dx \\
&=&C\int_{U}\varphi (x)E[\Phi (X_{T}^{x,n}) \\
&&\times \int_{0}^{T}u^{-H-\frac{1}{2}}\int_{u}^{T}a(s-u)(s-u)^{\frac{1}{2}%
-H}s^{H-\frac{1}{2}}\left( \frac{\partial }{\partial x}X_{s-u}^{x,n}\right)
^{\ast }dB_{s}du]^{\ast }dx \\
&=&I_{1}(n)+I_{2}(n),
\end{eqnarray*}%
where%
\begin{eqnarray*}
I_{1}(n) &:&=C\int_{U}\varphi (x)E[(\Phi (X_{T}^{x,n})-\Phi (X_{T}^{x})) \\
&&\times \int_{0}^{T}u^{-H-\frac{1}{2}}\int_{u}^{T}a(s-u)(s-u)^{\frac{1}{2}%
-H}s^{H-\frac{1}{2}}\left( \frac{\partial }{\partial x}X_{s-u}^{x,n}\right)
^{\ast }dB_{s}du]^{\ast }dx
\end{eqnarray*}%
and%
\begin{eqnarray*}
I_{2}(n) &:&=C\int_{U}\varphi (x)E[\Phi (X_{T}^{x})\int_{0}^{T}u^{-H-\frac{1%
}{2}} \\
&&\times \int_{u}^{T}a(s-u)(s-u)^{\frac{1}{2}-H}s^{H-\frac{1}{2}}\left( 
\frac{\partial }{\partial x}X_{s-u}^{x,n}\right) ^{\ast }dB_{s}du]^{\ast }dx
\\
&=&C\int_{U}\varphi (x)E[\Phi (X_{T}^{x})\int_{0}^{T}u^{-H-\frac{1}{2}} \\
&&\times \int_{u}^{T}a(s-u)(s-u)^{\frac{1}{2}-H}s^{H-\frac{1}{2}}\left( 
\frac{\partial }{\partial x}X_{s-u}^{x}\right) ^{\ast }dB_{s}du]^{\ast }dx \\
&&+I_{3}(n),
\end{eqnarray*}%
where%
\begin{eqnarray*}
&&I_{3}(n) \\
&:&=C\int_{U}\varphi (x)E[\Phi (X_{T}^{x})\int_{0}^{T}u^{-H-\frac{1}{2}%
}\int_{u}^{T}a(s-u)(s-u)^{\frac{1}{2}-H}s^{H-\frac{1}{2}} \\
&&\times \{\left( \frac{\partial }{\partial x}X_{s-u}^{x,n}\right) ^{\ast
}-\left( \frac{\partial }{\partial x}X_{s-u}^{x}\right) ^{\ast
}\}dB_{s}du]^{\ast }dx.
\end{eqnarray*}%
It follows from Fubini's theorem, H\"{o}lder's inequality, the It\^{o}
isometry, Lemma \ref{BoundedDerivatives}, Lemma \ref{ConvergenceX} and
dominated convergence that%
\begin{eqnarray*}
&&\left\Vert I_{1}(n)\right\Vert \\
&\leq &\left\Vert \varphi \right\Vert _{\infty }\int_{U}(E[\left\vert \Phi
(X_{T}^{x,n})-\Phi (X_{T}^{x})\right\vert ^{2}])^{1/2} \\
&&\times (\int_{0}^{T}s^{2H-1}E[(\int_{0}^{s}u^{-H-\frac{1}{2}}\left\vert
a(s-u)\right\vert (s-u)^{\frac{1}{2}-H}\left\Vert \frac{\partial }{\partial x%
}X_{s-u}^{x,n}\right\Vert du)^{2}]ds)^{1/2}dx \\
&\leq &\left\Vert \varphi \right\Vert _{\infty }\int_{U}(E[\left\vert \Phi
(X_{T}^{x,n})-\Phi (X_{T}^{x})\right\vert ^{2}])^{1/2}(\int_{0}^{T}s^{2H-1}
\\
&&\times \int_{0}^{s}\int_{0}^{s}u_{1}^{-H-\frac{1}{2}}\left\vert
a(s-u_{1})\right\vert (s-u_{1})^{\frac{1}{2}-H}u_{2}^{-H-\frac{1}{2}%
}\left\vert a(s-u_{2})\right\vert (s-u_{2})^{\frac{1}{2}-H} \\
&&\times E[\left\Vert \frac{\partial }{\partial x}X_{s-u_{1}}^{x,n}\right%
\Vert ^{2}]^{1/2}E[\left\Vert \frac{\partial }{\partial x}%
X_{s-u_{2}}^{x,n}\right\Vert ^{2}]^{1/2}du_{1}du_{2}ds)^{1/2}dx \\
&=&\left\Vert \varphi \right\Vert _{\infty }\int_{U}(E[\left\vert \Phi
(X_{T}^{x,n})-\Phi (X_{T}^{x})\right\vert ^{2}])^{1/2}(\int_{0}^{T}s^{2H-1}
\\
&&\times (\int_{0}^{s}u^{-H-\frac{1}{2}}\left\vert a(s-u)\right\vert (s-u)^{%
\frac{1}{2}-H}E[\left\Vert \frac{\partial }{\partial x}X_{s-u}^{x,n}\right%
\Vert ^{2}]^{1/2}du)^{2}ds)^{1/2}dx \\
&\leq &\left\Vert \varphi \right\Vert _{\infty }\int_{U}(E[\left\vert \Phi
(X_{T}^{x,n})-\Phi (X_{T}^{x})\right\vert ^{2}])^{1/2}dx(\int_{0}^{T}s^{2H-1}
\\
&&\times \sup_{n\geq 1}C_{1,2H,d,T}(\left\Vert b_{n}\right\Vert _{L_{\infty
}^{\infty }},\left\Vert b_{n}\right\Vert _{L_{\infty
}^{1}})^{1/4}(\int_{0}^{s}u^{-H-\frac{1}{2}}\left\vert a(s-u)\right\vert
(s-u)^{\frac{1}{2}-H}du)^{2}ds)^{1/2} \\
&\leq &C\left\Vert \varphi \right\Vert _{\infty }\int_{U}(E[\left\vert \Phi
(X_{T}^{x,n})-\Phi (X_{T}^{x})\right\vert ^{2}])^{1/2}dx(\int_{0}^{T}s^{H-%
\frac{1}{2}}ds)^{1/2} \\
&&\underset{n\longrightarrow \infty }{\longrightarrow }0,
\end{eqnarray*}%
where used the boundedness of the function $a$ in the last estimate.

By applying the Clark-Ocone formula (see e.g. \cite{Nualart}) in combination
with It\^{o}'s isometry and the chain rule for the Malliavin derivative, we
see that%
\begin{eqnarray*}
&&I_{3}(n) \\
&:&=C\int_{U}\varphi (x)E[E[\Phi (X_{T}^{x})]\int_{0}^{T}u^{-H-\frac{1}{2}%
}\int_{u}^{T}a(s-u)(s-u)^{\frac{1}{2}-H}s^{H-\frac{1}{2}} \\
&&\times \{\left( \frac{\partial }{\partial x}X_{s-u}^{x,n}\right) ^{\ast
}-\left( \frac{\partial }{\partial x}X_{s-u}^{x}\right) ^{\ast
}\}dB_{s}du]^{\ast }dx \\
&&+C\int_{U}\varphi (x)E[\int_{0}^{T}u^{-H-\frac{1}{2}%
}\int_{u}^{T}a(s-u)(s-u)^{\frac{1}{2}-H}s^{H-\frac{1}{2}}D_{s}\Phi
(X_{T}^{x}) \\
&&\times \{\frac{\partial }{\partial x}X_{s-u}^{x,n}-\frac{\partial }{%
\partial x}X_{s-u}^{x}\}^{\ast }dsdu]^{\ast }dx \\
&=&C\int_{U}\varphi (x)E[\int_{0}^{T}u^{-H-\frac{1}{2}%
}\int_{u}^{T}a(s-u)(s-u)^{\frac{1}{2}-H}s^{H-\frac{1}{2}}\Phi ^{\shortmid
}(X_{T}^{x})D_{s}X_{T}^{x} \\
&&\times \{\frac{\partial }{\partial x}X_{s-u}^{x,n}-\frac{\partial }{%
\partial x}X_{s-u}^{x}\}^{\ast }dsdu]^{\ast }dx.
\end{eqnarray*}%
Then using Lemma \ref{WeakConvergence}, Lemma \ref{BoundMalliavinDerivative}
and dominated convergence in connection with Lemma \ref{BoundedDerivatives},
we find that%
\begin{equation*}
\left\Vert I_{3}(n)\right\Vert \underset{n\longrightarrow \infty }{%
\longrightarrow }0.
\end{equation*}%
Here we mention that $D_{\cdot }X_{T}^{\cdot }$ used above stands for a weak
limit of a subsequence of $D_{\cdot }X_{T}^{\cdot ,n},n\geq 1$ in $%
L^{2}([0,T]\times \Omega \times U)$ such that $D_{\cdot }X_{T}^{x}$ is a
representative of the Malliavin derivative of $X_{T}^{x}$ for almost all $x$
in $U$. The latter however is a consequence of Lemma 1.2.3 in \cite{Nualart}
in connection with Lemma \ref{ConvergenceX}, dominated convergence and the
bound (\ref{D}), which is independent of $x$.

Similarly, we also obtain that%
\begin{equation*}
-\int_{U}E[\Phi (X_{T}^{x,n})]\frac{\partial }{\partial x}\varphi (x)dx%
\underset{n\longrightarrow \infty }{\longrightarrow }\int_{U}E[\Phi
(X_{T}^{x})]\frac{\partial }{\partial x}\varphi (x)dx.
\end{equation*}%
So%
\begin{eqnarray*}
&&-\int_{U}E[\Phi (X_{T}^{x})]\frac{\partial }{\partial x}\varphi (x)dx \\
&=&C\int_{U}\varphi (x)E[\Phi (X_{T}^{x}) \\
&&\times \int_{0}^{T}u^{-H-\frac{1}{2}}\int_{u}^{T}a(s-u)(s-u)^{\frac{1}{2}%
-H}s^{H-\frac{1}{2}}\left( \frac{\partial }{\partial x}X_{s-u}^{x}\right)
^{\ast }dB_{s}ds]^{\ast }dx
\end{eqnarray*}

Finally, we can apply the monotone class theorem in connection with
dominated convergence and the Cauchy-Schwarz inequality and verify the
latter relation for Borel measurable functions $\Phi :\mathbb{R}%
^{d}\longrightarrow \mathbb{R}$ such that 
\begin{equation*}
\Phi (X_{T}^{\cdot })\in L^{2}(\Omega \times U,\mu \times dx).
\end{equation*}%
Hence the result follows.
\end{proof}

\bigskip

In financial applications the right hand side of relation (\ref{BEL}), say $%
M $ may be interpreted as a sensitivity measure- known as delta- for changes
of the fair value of an option with payoff function $\Phi $ and underlying $%
d $ stock price processes $X_{\cdot }^{x}$ (under a change of measure) with
respect to the initial prices $x\in \mathbb{R}^{d}$ of the stocks. The
quantity $M$, is a priori for $H<\frac{1}{2(d+2)}$ only defined for \emph{%
almost all} initial values $x$. In practice, however, where a trader is
interested in sensitivities with respect to specific initial prices of the
stocks, the choice of $M$ as a sensitivity measure would be not
satisfactory. On the other hand, in order to make sense of $M$ as a delta
for \emph{all} $x\in \mathbb{R}^{d}$, one can in fact choose a version of $M$%
, which is continuous and hence defined for all $x$. It turns out, however
that such a version of $M$ exists, if the Hurst parameter is allowed to be a
little bit smaller than in Theorem \ref{Bismut}, that is $H<$ $\frac{1}{%
2(d+3)}.$ See Proposition \ref{ContinuousBEL}, whose proof requires the
following new estimate, which is based on Theorem \ref{mainestimate2}:

\begin{lemma}
\bigskip \label{SecondBoundedDerivative}Let $b\in C_{c}^{\infty
}((0,T)\times \mathbb{R}^{d}).$ Fix integers $p\geq 2$. Then, if $H<\frac{1}{%
2(d+3)},$ we have 
\begin{equation*}
\sup_{x\in \mathbb{R}^{d}}E[\left\Vert \frac{\partial ^{2}}{\partial x^{2}}%
X_{t}^{s,x}\right\Vert ^{p}]\leq C_{p,H,d,T}(\left\Vert b\right\Vert
_{L_{\infty }^{\infty }},\left\Vert b\right\Vert _{L_{\infty }^{1}})<\infty
\end{equation*}%
for some continuous function $C_{p,H,d,T}:[0,\infty )^{2}\longrightarrow
\lbrack 0,\infty )$.
\end{lemma}

\begin{proof}
Since \ the stochastic flow associated with the smooth vector field $b$ is
smooth, too (compare to e.g. \cite{Kunita}), we obtain that%
\begin{equation}
\frac{\partial }{\partial x}X_{t}^{s,x}=I_{d\times
d}+\int_{s}^{t}Db(u,X_{u}^{s,x})\frac{\partial }{\partial x}X_{u}^{s,x}du,
\label{InitialDerivative}
\end{equation}%
where $Db:\mathbb{R}^{d}\longrightarrow L(\mathbb{R}^{d},\mathbb{R}^{d})$ is
the derivative of $b$ with respect to the space variable.

Using Picard iteration, we see that%
\begin{equation}
\frac{\partial }{\partial x}X_{t}^{s,x}=I_{d\times d}+\sum_{m\geq
1}\int_{\Delta
_{s,t}^{m}}Db(u_{1},X_{u_{1}}^{s,x})...Db(u_{m},X_{u_{m}}^{s,x})du_{m}...du_{1},
\label{FirstOrderExpansion}
\end{equation}%
where%
\begin{equation*}
\Delta _{s,t}^{m}=\{(u_{m},...u_{1})\in \lbrack 0,T]^{m}:\theta
<u_{m}<...<u_{1}<t\}.
\end{equation*}

By using dominated convergence, we can differentiate both sides with respect
to $x$ and get that%
\begin{equation*}
\frac{\partial ^{2}}{\partial x^{2}}X_{t}^{s,x}=\sum_{m\geq 1}\int_{\Delta
_{s,t}^{m}}\frac{\partial }{\partial x}%
[Db(u_{1},X_{u_{1}}^{s,x})...Db(u_{m},X_{u_{m}}^{s,x})]du_{m}...du_{1}.
\end{equation*}%
Then application of the Leibniz and chain rule yields%
\begin{eqnarray*}
&&\frac{\partial }{\partial x}%
[Db(u_{1},X_{u_{1}}^{s,x})...Db(u_{m},X_{u_{m}}^{s,x})] \\
&=&\sum_{r=1}^{m}Db(u_{1},X_{u_{1}}^{s,x})...D^{2}b(u_{r},X_{u_{r}}^{s,x})%
\frac{\partial }{\partial x}X_{u_{r}}^{s,x}...Db(u_{m},X_{u_{m}}^{s,x}),
\end{eqnarray*}%
where $D^{2}b=D(Db):\mathbb{R}^{d}\longrightarrow L(\mathbb{R}^{d},L(\mathbb{%
R}^{d},\mathbb{R}^{d}))$.

So it follows from (\ref{FirstOrderExpansion}) that%
\begin{eqnarray}
\frac{\partial ^{2}}{\partial x^{2}}X_{t}^{s,x} &=&\sum_{m_{1}\geq
1}\int_{\Delta
_{s,t}^{m_{1}}}%
\sum_{r=1}^{m_{1}}Db(u_{1},X_{u_{1}}^{s,x})...D^{2}b(u_{r},X_{u_{r}}^{s,x}) 
\notag \\
&&\times \left( I_{d\times d}+\sum_{m_{2}\geq 1}\int_{\Delta
_{s,u_{r}}^{m_{2}}}Db(v_{1},X_{v_{1}}^{s,x})...Db(v_{m_{2}},X_{v_{m_{2}}}^{s,x})dv_{m_{2}}...dv_{1}\right)
\notag \\
&&\times
Db(u_{r+1},X_{u_{r+1}}^{s,x})...Db(u_{m_{1}},X_{u_{m_{1}}}^{s,x})du_{m_{1}}...du_{1}
\notag \\
&=&\sum_{m_{1}\geq 1}\sum_{r=1}^{m_{1}}\int_{\Delta
_{s,t}^{m_{1}}}Db(u_{1},X_{u_{1}}^{s,x})...D^{2}b(u_{r},X_{u_{r}}^{s,x})...Db(u_{m_{1}},X_{u_{m_{1}}}^{s,x})du_{m_{1}}...du_{1}
\notag \\
&&+\sum_{m_{1}\geq 1}\sum_{r=1}^{m_{1}}\sum_{m_{2}\geq 1}\int_{\Delta
_{s,t}^{m_{1}}}\int_{\Delta
_{s,u_{r}}^{m_{2}}}Db(u_{1},X_{u_{1}}^{s,x})...D^{2}b(u_{r},X_{u_{r}}^{s,x})
\notag \\
&&\times
Db(v_{1},X_{v_{1}}^{s,x})...Db(v_{m_{2}},X_{v_{m_{2}}}^{s,x})Db(u_{r+1},X_{u_{r+1}}^{s,x})...Db(u_{m_{1}},X_{u_{m_{1}}}^{s,x})
\notag \\
&&dv_{m_{2}}...dv_{1}du_{m_{1}}...du_{1}  \notag \\
&=&:I_{1}+I_{2}.  \label{SecondOrderExpansion}
\end{eqnarray}

We now aim at applying Lemma \ref{partialshuffle} to the term $I_{2}$ in (%
\ref{SecondOrderExpansion}) and find that%
\begin{equation}
I_{2}=\sum_{m_{1}\geq 1}\sum_{r=1}^{m_{1}}\sum_{m_{2}\geq 1}\int_{\Delta
_{s,t}^{m_{1}+m_{2}}}\mathcal{H}%
_{m_{1}+m_{2}}^{X}(u)du_{m_{1}+m_{2}}...du_{1}  \label{I2}
\end{equation}%
for $u=(u_{1},...,u_{m_{1}+m_{2}}),$ where the integrand $\mathcal{H}%
_{m_{1}+m_{2}}^{X}(u)\in \mathbb{R}^{d}\otimes \mathbb{R}^{d}\otimes \mathbb{%
R}^{d}$ possesses entries given by sums of at most $C(d)^{m_{1}+m_{2}}$
summands, which are products of length $m_{1}+m_{2}$ of functions belonging
to the class%
\begin{equation*}
\left\{ \frac{\partial ^{j}}{\partial x_{l_{1}}\partial x_{l_{j}}}%
b^{(i)}(u,X_{u}^{s,x}),j=1,2,l_{1},l_{2},i=1,...,d\right\} .
\end{equation*}%
Here it is crucial to mention that second order derivatives of functions in
those products of functions on $\Delta _{s,t}^{m_{1}+m_{2}}$ in (\ref{I2} )
only appear once. Thus the absolute value of the multi-index $\alpha $ with
respect to the total order of derivatives of those products of functions in
connection with Proposition \ref{OrderDerivatives} in the Appendix is given
by%
\begin{equation}
\left\vert \alpha \right\vert =m_{1}+m_{2}+1.  \label{alpha}
\end{equation}%
We now choose $p,c,r\in \lbrack 1,\infty )$ such that $cp=2^{q}$ for some
integer $q$ and $\frac{1}{r}+\frac{1}{c}=1.$ Then we can employ H\"{o}lder's
inequality and Girsanov's theorem (Theorem \ref{girsanov}) in combination
with Lemma \ref{Novikov} in the Appendix and get that%
\begin{eqnarray}
&&E[\left\Vert I_{2}\right\Vert ^{p}]  \notag \\
&\leq &C(\left\Vert b\right\Vert _{L_{\infty }^{\infty }})\left(
\sum_{m_{1}\geq 1}\sum_{r=1}^{m_{1}}\sum_{m_{2}\geq 1}\sum_{i\in
I}\left\Vert \int_{\Delta _{s,t}^{m_{1}+m_{2}}}\mathcal{H}%
_{i}^{B^{H}}(u)du_{m_{1}+m_{2}}...du_{1}\right\Vert _{L^{2^{q}}(\Omega ;%
\mathbb{R})}\right) ^{p},  \label{lp}
\end{eqnarray}%
where $C:[0,\infty )\longrightarrow \lbrack 0,\infty )$ is a continuous
function. Here $\#I\leq K^{m_{1}+m_{2}}$ for a constant $K=K(d)$ and the
integrands $\mathcal{H}_{i}^{B^{H}}(u)$ are of the form 
\begin{equation*}
\mathcal{H}_{i}^{B^{H}}(u)=\dprod%
\limits_{l=1}^{m_{1}+m_{2}}h_{l}(u_{l}),h_{l}\in \Lambda ,l=1,...,m_{1}+m_{2}
\end{equation*}%
where 
\begin{equation*}
\Lambda :=\left\{ \frac{\partial ^{j}}{\partial x_{l_{1}}\partial x_{l_{j}}}%
b^{(i)}(u,x+B_{u}^{H}),j=1,2,l_{1},l_{2},i=1,...,d\right\} .
\end{equation*}%
Also in this case functions with second order derivatives only appear once
in those products.

Define 
\begin{equation*}
J=\left( \int_{\Delta _{s,t}^{m_{1}+m_{2}}}\mathcal{H}%
_{i}^{B^{H}}(u)du_{m_{1}+m_{2}}...du_{1}\right) ^{2^{q}}.
\end{equation*}%
Using againLemma \ref{partialshuffle} in the Appendix, successively, we
obtain that $J$ can be written as a sum of, at most of length $%
K(q)^{m_{1}+m_{2}}$ with summands of the form%
\begin{equation}
\int_{\Delta
_{s,t}^{2^{q}(m_{1}+m_{2})}}\dprod%
\limits_{l=1}^{2^{q}(m_{1}+m_{2})}f_{l}(u_{l})du_{2^{q}(m_{1}+m_{2})}...du_{1},
\label{f}
\end{equation}%
where $f_{l}\in \Lambda $ for all $l$.

Here the number of factors $f_{l}$ in the above product, which have a second
order derivative, is exactly $2^{q}.$ Thus the total order of the
derivatives involved in (\ref{f}) in connection with Proposition \ref%
{OrderDerivatives} is given by%
\begin{equation}
\left\vert \alpha \right\vert =2^{q}(m_{1}+m_{2}+1).  \label{alpha2}
\end{equation}

We can now invoke Theorem \ref{mainestimate2} for $m=2^{q}(m_{1}+m_{2})$ and 
$\varepsilon _{j}=0$ and find that%
\begin{eqnarray*}
&&\left\vert E\left[ \int_{\Delta
_{s,t}^{2^{q}(m_{1}+m_{2})}}\dprod%
\limits_{l=1}^{2^{q}(m_{1}+m_{2})}f_{l}(u_{l})du_{2^{q}(m_{1}+m_{2})}...du_{1}%
\right] \right\vert \\
&\leq &C^{m_{1}+m_{2}}(\left\Vert b\right\Vert _{L^{1}(\mathbb{R}%
^{d};L^{\infty }([0,T];\mathbb{R}^{d}))})^{2^{q}(m_{1}+m_{2})} \\
&&\times \frac{((2(2^{q}(m_{1}+m_{2}+1))!)^{1/4}}{\Gamma
(-H(2d2^{q}(m_{1}+m_{2})+42^{q}(m_{1}+m_{2}+1))+22^{q}(m_{1}+m_{2}))^{1/2}}
\end{eqnarray*}%
for a constant $C$ depending on $H,T,d$ and $q$.

So the latter combined with (\ref{lp}) shows that%
\begin{eqnarray*}
&&E[\left\Vert I_{2}\right\Vert ^{p}] \\
&\leq &C(\left\Vert b\right\Vert _{L_{\infty }^{\infty }})(\sum_{m_{1}\geq
1}\sum_{m_{2}\geq 1}K^{m_{1}+m_{2}}((\left\Vert b\right\Vert _{L^{1}(\mathbb{%
R}^{d};L^{\infty }([0,T];\mathbb{R}^{d}))})^{2^{q}(m_{1}+m_{2})} \\
&&\times \frac{((2(2^{q}(m_{1}+m_{2}+1))!)^{1/4}}{\Gamma
(-H(2d2^{q}(m_{1}+m_{2})+42^{q}(m_{1}+m_{2}+1))+22^{q}(m_{1}+m_{2}))^{1/2}}%
)^{1/2^{q}})^{p}
\end{eqnarray*}%
for a constant $K$ depending on $H,T,d,p$ and $q$.

Since $\frac{1}{2(d+3)}\leq \frac{1}{2(d+2\frac{m_{1}+m_{2}+1}{m_{1}+m_{2}})}
$ for $m_{1},m_{2}\geq 1$, the above sum converges, when $H<\frac{1}{2(d+3)}$%
.

Further, one establishes in the same way a similar estimate for $%
E[\left\Vert I_{1}\right\Vert ^{p}]$. Altogether, the proof follows.
\end{proof}

Using Lemma \ref{SecondBoundedDerivative}, we can obtain the following
result:

\begin{theorem}
\label{2Sobolev}\bigskip Let $b\in L_{\infty ,\infty }^{1,\infty }$, $H<%
\frac{1}{2(d+3)}$ and $U\subset \mathbb{R}^{d}$ a bounded and open set. Then
for all $0\leq t\leq T$ we have that 
\begin{equation*}
X_{t}^{\cdot }\in \dbigcap\limits_{p\geq 2}L^{2}(\Omega ;W^{2,p}(U))\text{.}
\end{equation*}%
In particular, for all $0\leq t\leq T$ there exists a $\Omega ^{\ast }$ with 
$\mu (\Omega ^{\ast })=1$ such that for all $\omega \in \Omega ^{\ast }$ $%
(x\longmapsto X_{t}^{x}(\omega ))$ has a continuous version on $U$.
\end{theorem}

\begin{proof}
Following the ideas of Proposition 4.2 in \cite{MMNPZ}, we approximate $b$
by a sequence of vector fields $b_{n}\in C_{c}^{\infty }((0,T)\times \mathbb{%
R}^{d}),n\geq 1$ in the sense of the conditions (\ref{C1}), (\ref{C2}), (\ref%
{C3}). Let $X_{\cdot }^{x,n},n\geq 1$ be the sequence of strong solutions to
(\ref{SDE}) associated with those functions. Let $\phi \in C_{c}^{\infty }(U;%
\mathbb{R}^{d})$ and define for fixed $t\in \lbrack 0,T]$ the sequence of
random variables%
\begin{equation*}
\left\langle X_{t}^{\cdot ,n},\phi \right\rangle :=\int_{U}\left\langle
X_{t}^{x,n},\phi \right\rangle _{\mathbb{R}^{d}}dx,n\geq 1
\end{equation*}%
By invoking similar arguments as in the proof of Proposition 4.2 \cite[%
Proposition 4.2]{MMNPZ}, which relies on a compactness criterion for square
integrable funtionals of Wiener processes (see \cite{DMN}), in combination
with the estimates of Lemma 5.6 in \cite{BLPP} one proves that there exists
a subsequence $n_{j},j\geq 1$ such that%
\begin{equation}
\left\langle X_{t}^{\cdot ,n_{j}},\phi \right\rangle \underset{%
j\longrightarrow \infty }{\longrightarrow }\left\langle X_{t}^{\cdot },\phi
\right\rangle  \label{nj}
\end{equation}%
in $L^{2}(\Omega )$ strongly for all $\phi \in C_{c}^{\infty }(U;\mathbb{R}%
^{d})$, where $X_{s}^{x},0\leq s\leq T$ is the strong solution of Theorem %
\ref{StrongSolution}. Note that we also have from Proposition \ref%
{ConvergenceX} that%
\begin{equation*}
X_{t}^{x,n}\underset{n\longrightarrow \infty }{\longrightarrow }X_{t}^{x}
\end{equation*}%
in $L^{2}(\Omega )$ strongly.

Further, one gets from Lemma \ref{SecondBoundedDerivative} that%
\begin{eqnarray*}
&&\sup_{n\geq 1}\left\Vert X_{t}^{\cdot ,n}\right\Vert _{L^{2}(\Omega
;W^{2,p}(U))}^{2} \\
&\leq &\sum_{i=0}^{2}(\int_{\mathcal{U}}\sup_{n\geq 1}E\left[ \left\Vert
D^{i}X_{t}^{x,n}\right\Vert ^{p}\right] dx)^{\frac{2}{p}}<\infty
\end{eqnarray*}%
for $H<\frac{1}{2(d+3)}.$

On the other hand, we know that $L^{2}(\Omega ;W^{2,p}(U))$ is a reflexive
space for $p>1.$ Hence there exists a subsequence $n_{j},j\geq 1$ such that%
\begin{equation*}
X_{t}^{\cdot ,n}\underset{j\longrightarrow \infty }{\longrightarrow }Y
\end{equation*}%
in $L^{2}(\Omega ;W^{2,p}(U))$ weakly. For simplicity, suppose $n_{j},j\geq
1 $ coincides with the subsequence in (\ref{nj}). In addition, we obtain for
all $A\in \mathcal{F},$ $\phi \in C_{c}^{\infty }(U;\mathbb{R}^{d}),\alpha
^{(1)}+...+\alpha ^{(d)}\leq 2$ with $\alpha ^{(i)}\in \mathbb{N}%
_{0},i=1,...,d$ that%
\begin{eqnarray*}
&&E\left[ 1_{A}\left\langle X_{t}^{\cdot ,n_{j}},\frac{\partial ^{\alpha
^{(1)}+...+\alpha ^{(d)}}}{\partial ^{\alpha ^{(1)}}x_{1}...\partial
^{\alpha ^{(d)}}x_{d}}\phi \right\rangle \right] \\
&=&(-1)^{\alpha ^{(1)}+...+\alpha ^{(d)}}E\left[ 1_{A}\left\langle \frac{%
\partial ^{\alpha ^{(1)}+...+\alpha ^{(d)}}}{\partial ^{\alpha
^{(1)}}x_{1}...\partial ^{\alpha ^{(d)}}x_{1}}X_{t}^{\cdot ,n_{j}},\phi
\right\rangle \right] \\
&&\underset{j\longrightarrow \infty }{\longrightarrow }(-1)^{\alpha
^{(1)}+...+\alpha ^{(d)}}E\left[ 1_{A}\left\langle \frac{\partial ^{\alpha
^{(1)}+...+\alpha ^{(d)}}}{\partial ^{\alpha ^{(1)}}x_{1}...\partial
^{\alpha ^{(d)}}x_{1}}Y,\phi \right\rangle \right] .
\end{eqnarray*}%
On the other hand (\ref{nj}) also implies that%
\begin{eqnarray*}
&&E\left[ 1_{A}\left\langle X_{t}^{\cdot ,n_{j}},\frac{\partial ^{\alpha
^{(1)}+...+\alpha ^{(d)}}}{\partial ^{\alpha ^{(1)}}x_{1}...\partial
^{\alpha ^{(d)}}x_{d}}\phi \right\rangle \right] \\
&&\underset{j\longrightarrow \infty }{\longrightarrow }E\left[
1_{A}\left\langle X_{t}^{\cdot },\frac{\partial ^{\alpha ^{(1)}+...+\alpha
^{(d)}}}{\partial ^{\alpha ^{(1)}}x_{1}...\partial ^{\alpha ^{(d)}}x_{d}}%
\phi \right\rangle \right] .
\end{eqnarray*}%
Hence $X_{t}^{\cdot }\in L^{2}(\Omega ;W^{2,p}(U))$ for all $p\geq 2$.
\end{proof}

\bigskip

Denoting by $M=M(x),x\in U$ the right hand side of relation (\ref{BEL}), we
prove that $M$ possesses a continuous version:

\begin{proposition}
\label{ContinuousBEL}Retain the conditions of Theorem \ref{2Sobolev}. Let $%
p>\max (d,4)$ and $\Phi :\mathbb{R}^{d}\longrightarrow \mathbb{R}$ be a
bounded continuous function. Then $M$ has a continuous version on $U$, which
is obtained by replacing in $M$ on the right hand side of (\ref{BEL}) the
derivative of the flow by a predictable version $\{Y_{t}^{\cdot },0\leq
t\leq T\}\in $ $L^{2}([0,T]\times \Omega ,dt\times \mu ;W^{1,p}(U))$ with $%
Y_{t}^{\cdot }(\omega )\in C(U)$ for all $(t,\omega )$.
\end{proposition}

\begin{proof}
As before denote by $\mathcal{P}$ the predictable $\sigma -$algebra on $%
[0,T]\times \Omega $ with respect to $\{\mathcal{F}_{t}\}_{0\leq t\leq T}.$
Then, by using almost the same proof of Theorem \ref{2Sobolev} combined with
Lemma \ref{SecondBoundedDerivative}, one shows that there exists a $\frac{%
\partial }{\partial x}X_{\cdot }^{\cdot }\in L^{2}([0,T]\times \Omega ,%
\mathcal{P},dt\times \mu ;W^{1,p}(U))$ with $p>\max (d,4)$ such that $\frac{%
\partial }{\partial x}X_{t}^{\cdot }(\omega )$ is the Sobolev derivative of $%
X_{t}^{\cdot }(\omega )$ on $U$ $(t,\omega )-$a.e. So in particular, we see
for $\phi \in L^{\infty }(U;\mathbb{R})$ that 
\begin{equation*}
\int_{U}\frac{\partial }{\partial x}X_{t}^{x}\phi (x)dx,0\leq t\leq T
\end{equation*}%
is a predictable process. Now let us choose a continuous version $%
Y_{t}^{\cdot }(\omega )$ of $\frac{\partial }{\partial x}X_{t}^{\cdot
}(\omega )$ for all $(t,\omega )$ (which exists by a classical Sobolev space
theory and our assumptions). Then the process%
\begin{equation*}
\int_{U}Y_{t}^{x}(\omega )\phi (x)dx,0\leq t\leq T
\end{equation*}%
is predictable, too. Let $\delta _{\epsilon ,y}\in L^{\infty }(\mathbb{R}%
^{d}),\epsilon >0$ be an approximation of the Dirac delta measure in $y\in
U. $ Further, let $V$ be an open and bounded set with $\overline{V}\subset U$
and $y\in V$. In addition, consider a continuous function $\varsigma $ on $U$
with compact support in $U$ such that $\varsigma (x)=1$ for all $x\in V$.
Then 
\begin{equation*}
\int_{U}Y_{t}^{x}(\omega )\varsigma (x)\delta _{\epsilon ,y}(x)dx\underset{%
\epsilon \searrow 0}{\longrightarrow }Y_{t}^{y}(\omega )
\end{equation*}%
for all $(t,\omega )$. So $Y_{t}^{y},0\leq t\leq T$ is a predictable process
for all $y\in U$.

Using It\^{o}'s isometry we find that 
\begin{eqnarray*}
&&\sup_{x\in U}E[\left\Vert \int_{0}^{T}u^{-H-\frac{1}{2}%
}\int_{u}^{T}a(s)(s-u)^{\frac{1}{2}-H}s^{H-\frac{1}{2}}\left(
Y_{s-u}^{x}\right) ^{\ast }dB_{s}du\right\Vert ^{2}] \\
&=&C\sup_{x\in U}E[\int_{0}^{T}\left\Vert \int_{0}^{T}u^{-H-\frac{1}{2}}\chi
_{(u,T)}(s)a(s)(s-u)^{\frac{1}{2}-H}s^{H-\frac{1}{2}}\left(
Y_{s-u}^{x}\right) ^{\ast }du\right\Vert ^{2}ds] \\
&\leq &C\sup_{x\in U}\int_{0}^{T}E[\left( \int_{0}^{T}u^{-H-\frac{1}{2}}\chi
_{(u,T)}(s)\left\vert a(s)\right\vert (s-u)^{\frac{1}{2}-H}s^{H-\frac{1}{2}%
}\left\Vert Y_{s-u}^{x}\right\Vert du\right) ^{2}]ds.
\end{eqnarray*}%
On the other hand, we see that%
\begin{eqnarray}
&&E[\left( \int_{0}^{T}u^{-H-\frac{1}{2}}\chi _{(u,T)}(s)\left\vert
a(s)\right\vert (s-u)^{\frac{1}{2}-H}s^{H-\frac{1}{2}}\left\Vert
Y_{s-u}^{x}\right\Vert du\right) ^{2}  \notag \\
&=&\int_{0}^{T}\int_{0}^{T}u_{1}^{-H-\frac{1}{2}}\chi
_{(u_{1},T)}(s)\left\vert a(s)\right\vert (s-u_{1})^{\frac{1}{2}-H}s^{H-%
\frac{1}{2}}  \notag \\
&&\times u_{2}^{-H-\frac{1}{2}}\chi _{(u_{2},T)}(s)\left\vert
a(s)\right\vert (s-u_{2})^{\frac{1}{2}-H}s^{H-\frac{1}{2}}  \notag \\
&&\times E[\left\Vert Y_{s-u_{1}}^{x}\right\Vert \left\Vert
Y_{s-u_{2}}^{x}\right\Vert ]du_{1}du_{2}  \notag \\
&\leq &\int_{0}^{T}\int_{0}^{T}u_{1}^{-H-\frac{1}{2}}\chi
_{(u_{1},T)}(s)\left\vert a(s)\right\vert (s-u_{1})^{\frac{1}{2}-H}s^{H-%
\frac{1}{2}}  \notag \\
&&\times u_{2}^{-H-\frac{1}{2}}\chi _{(u_{2},T)}(s)\left\vert
a(s)\right\vert (s-u_{2})^{\frac{1}{2}-H}s^{H-\frac{1}{2}}  \notag \\
&&\times E[\left\Vert Y_{s-u_{1}}^{x}\right\Vert ^{2}]^{1/2}E[\left\Vert
Y_{s-u_{2}}^{x}\right\Vert ^{2}]^{1/2}du_{1}du_{2}.  \label{E}
\end{eqnarray}%
Let $b_{n},n\geq 1$ be a sequence of smooth functions, which approximates $b$
in the sense of Theorem \ref{2Sobolev}.\ Denote by $X_{\cdot }^{x,n},n\geq 1$
the corresponding solutions. Then it follows from Lemma \ref%
{BoundedDerivatives} that for all $B\in \mathcal{B}([0;T]),$ $G\in \mathcal{B%
}(U):$ 
\begin{eqnarray*}
&&\int_{B}\int_{G}E[\left\langle \frac{\partial }{\partial x}X_{t}^{x,n},%
\frac{\partial }{\partial x}X_{t}^{x,m}\right\rangle _{\mathbb{R}^{d\times
d}}]dxdt \\
&\leq &\left\vert \int_{B}\int_{G}E[\left\langle \frac{\partial }{\partial x}%
X_{t}^{x,n},\frac{\partial }{\partial x}X_{t}^{x,m}\right\rangle _{\mathbb{R}%
^{d\times d}}]dxdt\right\vert \\
&\leq &\int_{B}\int_{G}E[\left\Vert \frac{\partial }{\partial x}%
X_{t}^{x,n}\right\Vert ^{2}]^{1/2}E[\left\Vert \frac{\partial }{\partial x}%
X_{t}^{x,m}\right\Vert ^{2}]^{1/2}dxdt \\
&\leq &\int_{B}\int_{G}(C_{2,H,d,T}(\left\Vert b_{n}\right\Vert _{L_{\infty
}^{\infty }},\left\Vert b_{n}\right\Vert _{L_{\infty
}^{1}}))^{1/2}(C_{2,H,d,T}(\left\Vert b_{m}\right\Vert _{L_{\infty }^{\infty
}},\left\Vert b_{m}\right\Vert _{L_{\infty }^{1}}))^{1/2}dxdt \\
&\leq &\int_{B}\int_{G}Kdxdt,
\end{eqnarray*}%
where $K<\infty $ is a constant only depending on $H,d,T$ and the "size" of $%
b$. Hence, by using Lemma \ref{BoundedDerivatives} and weak convergence both
in $L^{2}([0,T]\times \Omega \times U,dt\times \mu \times dx;\mathbb{R}%
^{d\times d})$ and $L^{2}([0,T]\times \Omega ,\mathcal{P},dt\times \mu
;W^{1,p}(U))$ for suitable subsequences with respect to $n$ and $m$,
successively, we see that $t-$a.e, $x-$a.e. 
\begin{equation*}
E[\left\Vert Y_{t}^{x}\right\Vert ^{2}]\leq C.
\end{equation*}%
Using Fatou's Lemma combined with the continuity of $(x\longmapsto
Y_{t}^{x}(\omega ))$ for all $(t,\omega ),$ we also find that $t-$a.e.%
\begin{equation}
E[\left\Vert Y_{t}^{x}\right\Vert ^{2}]\leq C  \label{Y}
\end{equation}%
for all $x\in U$. Similarly, one shows that 
\begin{equation}
E[\left\Vert Y_{t}^{x}\right\Vert ^{4}]\leq C  \label{Z}
\end{equation}%
for a constant $C<\infty $ for all $x\in U$. So we obtain from (\ref{E}) that%
\begin{eqnarray}
&&\sup_{x\in U}E[\left( \int_{0}^{T}u^{-H-\frac{1}{2}}\chi
_{(u,T)}(s)\left\vert a(s)\right\vert (s-u)^{\frac{1}{2}-H}s^{H-\frac{1}{2}%
}\left\Vert Y_{s-u}^{x}\right\Vert du\right) ^{2}]  \notag \\
&\leq &C\int_{0}^{T}\int_{0}^{T}u_{1}^{-H-\frac{1}{2}}\chi
_{(u_{1},T)}(s)(s-u_{1})^{\frac{1}{2}-H}s^{H-\frac{1}{2}}  \notag \\
&&\times u_{2}^{-H-\frac{1}{2}}\chi _{(u_{2},T)}(s)(s-u_{2})^{\frac{1}{2}%
-H}s^{H-\frac{1}{2}}du_{1}du_{2}  \notag \\
&=&C\int_{0}^{T}u^{-H-\frac{1}{2}}\int_{u}^{T}(s-u)^{\frac{1}{2}-H}s^{H-%
\frac{1}{2}}dsdu<\infty .  \label{sup}
\end{eqnarray}%
So 
\begin{equation}
\sup_{x\in U}E[\left\Vert \int_{0}^{T}u^{-H-\frac{1}{2}%
}\int_{u}^{T}a(s)(s-u)^{\frac{1}{2}-H}s^{H-\frac{1}{2}}\left(
Y_{s-u}^{x}\right) ^{\ast }dB_{s}du\right\Vert ^{2}]<\infty .  \label{2nd}
\end{equation}%
Now let $x_{m}\underset{m\longrightarrow \infty }{\longrightarrow }x\in U$.
Then%
\begin{eqnarray*}
&&\left\Vert E[\Phi (Y_{T}^{x_{m}})\int_{0}^{T}u^{-H-\frac{1}{2}%
}\int_{u}^{T}a(s)(s-u)^{\frac{1}{2}-H}s^{H-\frac{1}{2}}\left(
Y_{s-u}^{x_{m}}\right) ^{\ast }dB_{s}du\right. \\
&&\left. -\Phi (Y_{T}^{x})\int_{0}^{T}u^{-H-\frac{1}{2}%
}\int_{u}^{T}a(s)(s-u)^{\frac{1}{2}-H}s^{H-\frac{1}{2}}\left(
Y_{s-u}^{x}\right) ^{\ast }dB_{s}du]\right\Vert \\
&\leq &\left\Vert I_{1}\right\Vert +\left\Vert I_{2}\right\Vert ,
\end{eqnarray*}%
where 
\begin{eqnarray*}
&&I_{1} \\
&:&=E[(\Phi (Y_{T}^{x_{m}})-\Phi (Y_{T}^{x}))\int_{0}^{T}u^{-H-\frac{1}{2}%
}\int_{u}^{T}a(s)(s-u)^{\frac{1}{2}-H}s^{H-\frac{1}{2}}\left(
Y_{s-u}^{x_{m}}\right) ^{\ast }dB_{s}du]
\end{eqnarray*}%
and%
\begin{eqnarray*}
&&I_{2} \\
&:&=E[\Phi (Y_{T}^{x})\int_{0}^{T}u^{-H-\frac{1}{2}}\int_{u}^{T}a(s)(s-u)^{%
\frac{1}{2}-H}s^{H-\frac{1}{2}}(\left( Y_{s-u}^{x_{m}}\right) ^{\ast
}-\left( Y_{s-u}^{x}\right) ^{\ast })dB_{s}du].
\end{eqnarray*}%
It follows from It\^{o}'s isometry and (\ref{sup}) that%
\begin{eqnarray*}
&&\left\Vert I_{1}\right\Vert \\
&\leq &E[(\Phi (Y_{T}^{x_{m}})-\Phi (Y_{T}^{x}))^{2}]^{1/2} \\
&&\times (\sup_{x\in U}E[\left( \int_{0}^{T}u^{-H-\frac{1}{2}}\chi
_{(u,T)}(s)\left\vert a(s)\right\vert (s-u)^{\frac{1}{2}-H}s^{H-\frac{1}{2}%
}\left\Vert Y_{s-u}^{x}\right\Vert du\right) ^{2}])^{1/2} \\
&\leq &CE[(\Phi (Y_{T}^{x_{m}})-\Phi (Y_{T}^{x}))^{2}]^{1/2}.
\end{eqnarray*}%
So because of dominated convergence $I_{1}=I_{1}(m)\underset{%
m\longrightarrow \infty }{\longrightarrow }0.$

On the other hand%
\begin{eqnarray*}
&&\left\Vert I_{2}\right\Vert \\
&\leq &CE[\left\Vert \int_{0}^{T}u^{-H-\frac{1}{2}}\int_{u}^{T}a(s)(s-u)^{%
\frac{1}{2}-H}s^{H-\frac{1}{2}}(\left( Y_{s-u}^{x_{m}}\right) ^{\ast
}-\left( Y_{s-u}^{x}\right) ^{\ast })dB_{s}du\right\Vert ^{2}]^{1/2} \\
&\leq &\sup_{x\in U}E[\left( \int_{0}^{T}u^{-H-\frac{1}{2}}\chi
_{(u,T)}(s)\left\vert a(s)\right\vert (s-u)^{\frac{1}{2}-H}s^{H-\frac{1}{2}%
}\left\Vert Y_{s-u}^{x_{m}}-Y_{s-u}^{x}\right\Vert du\right) ^{2}] \\
&&\int_{0}^{T}\int_{0}^{T}u_{1}^{-H-\frac{1}{2}}\chi
_{(u_{1},T)}(s)\left\vert a(s)\right\vert (s-u_{1})^{\frac{1}{2}-H}s^{H-%
\frac{1}{2}} \\
&&\times u_{2}^{-H-\frac{1}{2}}\chi _{(u_{2},T)}(s)\left\vert
a(s)\right\vert (s-u_{2})^{\frac{1}{2}-H}s^{H-\frac{1}{2}} \\
&&\times E[\left\Vert Y_{s-u_{1}}^{x}-Y_{s-u_{1}}^{x_{m}}\right\Vert
^{2}]^{1/2}E[\left\Vert Y_{s-u_{2}}^{x}-Y_{s-u_{2}}^{x_{m}}\right\Vert
^{2}]^{1/2}du_{1}du_{2}.
\end{eqnarray*}%
Because of continuity we know that%
\begin{equation*}
\left\Vert Y_{t}^{x}-Y_{t}^{x_{m}}\right\Vert ^{2}\underset{m\longrightarrow
\infty }{\longrightarrow }0
\end{equation*}%
for all $(t,\omega )$. So it follows from uniform integrability in
connection with (\ref{Z}) that%
\begin{equation*}
E[\left\Vert Y_{t}^{x}-Y_{t}^{x_{m}}\right\Vert ^{2}]\underset{%
m\longrightarrow \infty }{\longrightarrow }0
\end{equation*}%
for all $t$. Then using (\ref{Y}) and dominated convergence shows that%
\begin{equation*}
I_{2}=I_{2}(m)\underset{m\longrightarrow \infty }{\longrightarrow }0.
\end{equation*}%
In summary, we see that $M=M(x)$ with the derivative of the flow replaced by 
$Y_{t}^{\cdot },0\leq t\leq T$ is continuous in $x$ on $U$.
\end{proof}

\bigskip

\bigskip

\section{\protect\bigskip Application: Stock price model with stochastic
volatility}

In this Section we propose a model for stock prices $S_{t}^{x_{1},x_{2}},0%
\leq t\leq T$ with stochastic volatility $\sigma _{t}^{x_{2}},0\leq t\leq T$
described by the following SDE%
\begin{eqnarray}
S_{t}^{x_{1},x_{2}} &=&x_{1}+\int_{0}^{t}\mu
S_{u}^{x_{1},x_{2}}du+\int_{0}^{t}g(\sigma
_{u}^{x_{2}})S_{u}^{x_{1},x_{2}}dW_{u}  \notag \\
\sigma _{t}^{x_{2}} &=&x_{2}+\int_{0}^{t}b(u,\sigma
_{u}^{x_{2}})du+B_{t}^{H},x_{1},x_{2}\in \mathbb{R},0\leq t\leq T,
\label{Stock}
\end{eqnarray}%
where $W_{\cdot \text{ }}$ is a Wiener process, which is independent of a
fractional Brownian motion $B_{\cdot }^{H}$ with Hurst parameter $H<\frac{1}{%
2(d+2)}=\frac{1}{6},$ and where $\mu \in \mathbb{R}$, $b\in L_{\infty
,\infty }^{1,\infty }$ and $g:\mathbb{R}\longrightarrow (\alpha ,\infty )$
belongs to $C_{b}^{2}(\mathbb{R})$ for same $\alpha >0$. Let us also assume
that $\Omega =\Omega _{1}\times \Omega _{2}$ for sample spaces $\Omega _{1}$%
, $\Omega _{2}$, on which $W_{\cdot \text{ }}$ and $B_{\cdot }^{H}$ are
defined, respectively.

For a moment, let us assume that $b\in C_{c}^{\infty }((0,T)\times \mathbb{R}%
^{d}).$ Then $X_{t}^{x}:=(S_{t}^{x_{1},x_{2}},\sigma _{t}^{x_{2}})^{\ast
},x=(x_{1},x_{2})$ is Malliavin differentiable with respect to $%
Z=(Z^{(1)},Z^{(2)})^{\ast }=(W,B^{H})^{\ast }$ with Malliavin derivative $%
D=(D^{W},D^{H})^{\ast }$ and we get%
\begin{eqnarray*}
&&D_{s}X_{t}^{x} \\
&=&\int_{s}^{t}\left( 
\begin{array}{ll}
\mu & 0 \\ 
0 & b^{\shortmid }(u,\sigma _{u}^{x_{2}})%
\end{array}%
\right) D_{s}X_{u}^{x}du \\
&&+\left( \sum_{j=1}^{2}\int_{s}^{t}\sum_{l=1}^{2}\frac{\partial }{\partial
x_{l}}a_{ij}(S_{u}^{x_{1},x_{2}},\sigma
_{u}^{x_{2}})(D_{s}X_{u}^{x})_{rl}dZ_{u}^{(j)}\right) _{1\leq i,r\leq 2} \\
&&+\chi _{_{\lbrack 0,t]}(s)}\left( a_{ij}(S_{s}^{x_{1},x_{2}},\sigma
_{s}^{x_{2}})\right) _{1\leq i,j\leq 2} \\
&=&\int_{s}^{t}\left( 
\begin{array}{ll}
\mu & 0 \\ 
0 & b^{\shortmid }(u,\sigma _{u}^{x_{2}})%
\end{array}%
\right) D_{s}X_{u}^{x}du \\
&&+\left( \int_{s}^{t}\sum_{l=1}^{2}\frac{\partial }{\partial x_{l}}%
a_{i1}(S_{u}^{x_{1},x_{2}},\sigma
_{u}^{x_{2}})(D_{s}X_{u}^{x})_{rl}dW_{u}\right) _{1\leq i,r\leq 2} \\
&&+\chi _{_{\lbrack 0,t]}(s)}\left( a_{ij}(S_{s}^{x_{1},x_{2}},\sigma
_{s}^{x_{2}})\right) _{1\leq i,j\leq 2}
\end{eqnarray*}%
where 
\begin{equation*}
\left( a_{ij}(x_{1},x_{2})\right) _{1\leq i,j\leq 2}=\left( 
\begin{array}{ll}
g(x_{1})x_{2} & 0 \\ 
0 & 1%
\end{array}%
\right) .
\end{equation*}%
We know that $X_{t}^{x,y}$ is twice continuously differentiable with respect
to $(x,y)$. Then using a substitution formula for Wiener integrals \cite[%
proof of Theorem 3.2.9]{Nualart}, one finds similarly to the proof of Theorem \ref%
{Bismut} that

\begin{equation*}
D_{s}X_{t}^{x}=\frac{\partial }{\partial x}X_{t}^{s,X_{s}^{x}}\chi
_{_{\lbrack 0,t]}(s)}\left( a_{ij}(S_{s}^{x_{1},x_{2}},\sigma
_{s}^{x_{2}})\right) _{1\leq i,j\leq 2}.
\end{equation*}%
Similarly, we get for a payoff function $\Phi \in C_{c}^{\infty }(\mathbb{R}%
^{2})$ that%
\begin{equation*}
\frac{\partial }{\partial x}E[\Phi (X_{T}^{x,n})]=E[\Phi ^{\shortmid
}(X_{T}^{x})\frac{\partial }{\partial x}X_{T}^{s,X_{s}^{x}}\frac{\partial }{%
\partial x}X_{s}^{x}].
\end{equation*}%
So%
\begin{equation*}
\frac{\partial }{\partial x}E[\Phi (X_{T}^{x})]=E[\Phi ^{\shortmid
}(X_{T}^{x})D_{s}X_{T}^{x}\left( a_{ij}(S_{s}^{x_{1},x_{2}},\sigma
_{s}^{x_{2}})\right) _{1\leq i,j\leq 2}^{-1}\frac{\partial }{\partial x}%
X_{s}^{x}].
\end{equation*}%
Hence, for a bounded measurable function $a$ summing up to one we obtain by
means of the chain rule with respect to $D_{\cdot \text{ }}$ that%
\begin{eqnarray*}
&&\frac{\partial }{\partial x}E[\Phi (X_{T}^{x})] \\
&=&E[\int_{0}^{T}\{a(s)\Phi ^{\shortmid }(X_{T}^{x})D_{s}X_{T}^{x}\left(
a_{ij}(S_{s}^{x_{1},x_{2}},\sigma _{s}^{x_{2}})\right) _{1\leq i,j\leq
2}^{-1}\frac{\partial }{\partial x}X_{s}^{x}\}ds] \\
&=&E[\int_{0}^{T}\{a(s)D_{s}\Phi (X_{T}^{x})\left(
a_{ij}(S_{s}^{x_{1},x_{2}},\sigma _{s}^{x_{2}})\right) _{1\leq i,j\leq
2}^{-1}\frac{\partial }{\partial x}X_{s}^{x}\}ds]
\end{eqnarray*}%
We have that%
\begin{eqnarray*}
&&\left( a_{ij}(S_{s}^{x_{1},x_{2}},\sigma _{s}^{x_{2}})\right) _{1\leq
i,j\leq 2}^{-1}\frac{\partial }{\partial x}X_{s}^{x} \\
&=&\left( 
\begin{array}{ll}
(S_{s}^{x_{1},x_{2}}g(\sigma _{s}^{x_{2}}))^{-1}\frac{\partial }{\partial
x_{1}}S_{s}^{x_{1},x_{2}} & (S_{s}^{x_{1},x_{2}}g(\sigma _{s}^{x_{2}}))^{-1}%
\frac{\partial }{\partial x_{2}}S_{s}^{x_{1},x_{2}} \\ 
0 & \frac{\partial }{\partial x_{2}}\sigma _{s}^{x_{2}}%
\end{array}%
\right) .
\end{eqnarray*}%
Thus%
\begin{eqnarray*}
&&D_{s}\Phi (X_{T}^{x})\left( a_{ij}(S_{s}^{x_{1},x_{2}},\sigma
_{s}^{x_{2}})\right) _{1\leq i,j\leq 2}^{-1}\frac{\partial }{\partial x}%
X_{s}^{x} \\
&=&(D_{s}^{W}\Phi (X_{T}^{x})(S_{s}^{x_{1},x_{2}}g(\sigma _{s}^{x_{2}}))^{-1}%
\frac{\partial }{\partial x_{1}}S_{s}^{x_{1},x_{2}}, \\
&&D_{s}^{W}\Phi (X_{T}^{x})(S_{s}^{x_{1},x_{2}}g(\sigma _{s}^{x_{2}}))^{-1}%
\frac{\partial }{\partial x_{2}}S_{s}^{x_{1},x_{2}} \\
&&+D_{s}^{H}\Phi (X_{T}^{x})\frac{\partial }{\partial x_{2}}\sigma
_{s}^{x_{2}})^{\ast }.
\end{eqnarray*}%
So it follows that%
\begin{eqnarray*}
&&\frac{\partial }{\partial x}E[\Phi (X_{T}^{x})] \\
&=&(E[\int_{0}^{T}a(s)D_{s}^{W}\Phi (X_{T}^{x})(S_{s}^{x_{1},x_{2}}g(\sigma
_{s}^{x_{2}}))^{-1}\frac{\partial }{\partial x_{1}}S_{s}^{x_{1},x_{2}}ds], \\
&&E[\int_{0}^{T}a(s)D_{s}^{W}\Phi (X_{T}^{x})(S_{s}^{x_{1},x_{2}}g(\sigma
_{s}^{x_{2}}))^{-1}\frac{\partial }{\partial x_{2}}S_{s}^{x_{1},x_{2}}ds] \\
&&+E[\int_{0}^{T}a(s)D_{s}^{H}\Phi (X_{T}^{x})\frac{\partial }{\partial x_{2}%
}\sigma _{s}^{x_{2}}ds])^{\ast }.
\end{eqnarray*}%
In fact, using the independence of $W_{\cdot }$ and $B_{\cdot }^{H}$, we can
employ the proof of Theorem \ref{Bismut} and get that%
\begin{eqnarray*}
&&E[\int_{0}^{T}a(s)D_{s}^{H}\Phi (X_{T}^{x})\frac{\partial }{\partial x_{2}}%
\sigma _{s}^{x_{2}}ds] \\
&=&C_{H}E[\Phi (X_{T}^{x})\int_{0}^{T}u^{-H-\frac{1}{2}%
}\int_{u}^{T}a(s-u)(s-u)^{\frac{1}{2}-H}s^{H-\frac{1}{2}}\frac{\partial }{%
\partial x}\sigma _{s-u}^{x_{2}}dB_{s}du],
\end{eqnarray*}%
where $B_{\cdot }$ is a one-dimensional Brownian motion with respect to the
representation of $B_{\cdot }^{H}$ in (\ref{RepFractional}).

Finally, we can apply the duality formula with respect to $W_{\cdot }$ and
similar arguments as in the proof of Theorem \ref{Bismut} based on regular
functions $g,$ $b,$ $\Phi $ and we obtain the following BLE-formula for our
stock price model (\ref{Stock}):

\begin{theorem}
\label{Volatility}Let $U\subset \mathbb{R}^{2}$ be a bounded, open set and $%
b\in L_{\infty ,\infty }^{1,\infty }$ in the stock price model (\ref{Stock}%
). Further, assume that $g:\mathbb{R}\longrightarrow (\alpha ,\infty )$
belongs to $C_{b}^{2}(\mathbb{R})$ for some $\alpha >0$ and that $\Phi :%
\mathbb{R}^{2}\longrightarrow \mathbb{R}$ satisfies%
\begin{equation*}
\Phi (S_{T}^{\cdot ,\cdot },\sigma _{T}^{\cdot })\in L^{2}(\Omega \times
U,\mu \times dx)\text{.}
\end{equation*}%
In addition, let $a$ be a bounded and measurable function on $[0,T]$, which
sums up to $1$. Then%
\begin{eqnarray}
&&\frac{\partial }{\partial x}E[\Phi (S_{T}^{x_{1},x_{2}},\sigma
_{T}^{x_{2}})]  \notag \\
&=&(E[\Phi (X_{T}^{x})\int_{0}^{T}a(s)(S_{s}^{x_{1},x_{2}}g(\sigma
_{s}^{x_{2}}))^{-1}\frac{\partial }{\partial x_{1}}%
S_{s}^{x_{1},x_{2}}dW_{s}],  \notag \\
&&E[\Phi (X_{T}^{x})\int_{0}^{T}a(s)(S_{s}^{x_{1},x_{2}}g(\sigma
_{s}^{x_{2}}))^{-1}\frac{\partial }{\partial x_{2}}S_{s}^{x_{1},x_{2}}dW_{s}]
\notag \\
&&+C_{H}E[\Phi (X_{T}^{x})\int_{0}^{T}u^{-H-\frac{1}{2}%
}\int_{u}^{T}a(s-u)(s-u)^{\frac{1}{2}-H}s^{H-\frac{1}{2}}\frac{\partial }{%
\partial x}\sigma _{s-u}^{x_{2}}dB_{s}du])^{\ast }  \label{SBEL}
\end{eqnarray}%
for almost all  $x=(x_{1},x_{2})\in U$, where $C_{H}$ is a constant as given
in Theorem \ref{Bismut}.
\end{theorem}

\begin{remark}
If $H<\frac{1}{2(d+3)}=\frac{1}{8}$, one can show just as in Theorem \ref%
{ContinuousBEL} that the right hand side of (\ref{SBEL}) has a continuous
version.
\end{remark}

\bigskip

\section{Appendix}

\bigskip We want to recall here a version of Girsanov's theorem for the
fractional Brownian motion, which we need in the the proof of Lemma \ref%
{SecondBoundedDerivative}. For this purpose, let us pass in review some
basic concepts from fractional calculus (see \cite{samko.et.al.93} and \cite%
{lizorkin.01}).

Let $a,$ $b\in \mathbb{R}$ with $a<b$. Let $f\in L^{p}([a,b])$ with $p\geq 1$
and $\alpha >0$. Introduce the \emph{left-} and \emph{right-sided
Riemann-Liouville fractional integrals} as 
\begin{equation*}
I_{a^{+}}^{\alpha }f(x)=\frac{1}{\Gamma (\alpha )}\int_{a}^{x}(x-y)^{\alpha
-1}f(y)dy
\end{equation*}%
and 
\begin{equation*}
I_{b^{-}}^{\alpha }f(x)=\frac{1}{\Gamma (\alpha )}\int_{x}^{b}(y-x)^{\alpha
-1}f(y)dy
\end{equation*}%
for almost all $x\in \lbrack a,b]$, where $\Gamma $ is the Gamma function.

For a given integer $p\geq 1$, let $I_{a^{+}}^{\alpha }(L^{p})$ (resp. $%
I_{b^{-}}^{\alpha }(L^{p})$) be the image of $L^{p}([a,b])$ of the operator $%
I_{a^{+}}^{\alpha }$ (resp. $I_{b^{-}}^{\alpha }$). If $f\in
I_{a^{+}}^{\alpha }(L^{p})$ (resp. $f\in I_{b^{-}}^{\alpha }(L^{p})$) and $%
0<\alpha <1$ then we can define the \emph{left-} and \emph{right-sided
Riemann-Liouville fractional derivatives} by 
\begin{equation*}
D_{a^{+}}^{\alpha }f(x)=\frac{1}{\Gamma (1-\alpha )}\frac{d}{dx}\int_{a}^{x}%
\frac{f(y)}{(x-y)^{\alpha }}dy
\end{equation*}%
and 
\begin{equation*}
D_{b^{-}}^{\alpha }f(x)=\frac{1}{\Gamma (1-\alpha )}\frac{d}{dx}\int_{x}^{b}%
\frac{f(y)}{(y-x)^{\alpha }}dy.
\end{equation*}

The left- and right-sided derivatives of $f$ can be also represented as 
\begin{equation*}
D_{a^{+}}^{\alpha }f(x)=\frac{1}{\Gamma (1-\alpha )}\left( \frac{f(x)}{%
(x-a)^{\alpha }}+\alpha \int_{a}^{x}\frac{f(x)-f(y)}{(x-y)^{\alpha +1}}%
dy\right)
\end{equation*}%
and 
\begin{equation*}
D_{b^{-}}^{\alpha }f(x)=\frac{1}{\Gamma (1-\alpha )}\left( \frac{f(x)}{%
(b-x)^{\alpha }}+\alpha \int_{x}^{b}\frac{f(x)-f(y)}{(y-x)^{\alpha +1}}%
dy\right) .
\end{equation*}

Using the above definitions, one obtains that 
\begin{equation*}
I_{a^{+}}^{\alpha }(D_{a^{+}}^{\alpha }f)=f
\end{equation*}%
for all $f\in I_{a^{+}}^{\alpha }(L^{p})$ and 
\begin{equation*}
D_{a^{+}}^{\alpha }(I_{a^{+}}^{\alpha }f)=f
\end{equation*}%
for all $f\in L^{p}([a,b])$ and similarly for $I_{b^{-}}^{\alpha }$ and $%
D_{b^{-}}^{\alpha }$.

\bigskip

Let now $B^{H}=\{B_{t}^{H},t\in \lbrack 0,T]\}$ be a $d$-dimensional \emph{%
fractional Brownian motion} with Hurst parameter $H\in (0,1/2)$, that is $%
B^{H}$ is a centered Gaussian process with a covariance function given by 
\begin{equation*}
(R_{H}(t,s))_{i,j}:=E[B_{t}^{H,(i)}B_{s}^{H,(j)}]=\delta _{ij}\frac{1}{2}%
\left( t^{2H}+s^{2H}-|t-s|^{2H}\right) ,\quad i,j=1,\dots ,d,
\end{equation*}%
where $\delta _{ij}$ is one, if $i=j$, or zero else.

In the sequel we briefly recall the construction of the fractional Brownian
motion, which can be found in \cite{Nualart}. For simplicity, consider the
case $d=1$.

Let $\mathcal{E}$ be the set of step functions on $[0,T]$ and $\mathcal{H}$
be the Hilbert space given by the completion of $\mathcal{E}$ with respect
to the inner product 
\begin{equation*}
\langle 1_{[0,t]},1_{[0,s]}\rangle _{\mathcal{H}}=R_{H}(t,s).
\end{equation*}%
From that we get an extension of the mapping $1_{[0,t]}\mapsto B_{t}$ to an
isometry between $\mathcal{H}$ and a Gaussian subspace of $L^{2}(\Omega )$
with respect to $B^{H}$. We denote by $\varphi \mapsto B^{H}(\varphi )$ this
isometry.

If $H<1/2$, one shows that the covariance function $R_{H}(t,s)$ has the
representation

\bigskip\ 
\begin{equation}
R_{H}(t,s)=\int_{0}^{t\wedge s}K_{H}(t,u)K_{H}(s,u)du,  \label{2.2}
\end{equation}%
where 
\begin{equation}
K_{H}(t,s)=c_{H}\left[ \left( \frac{t}{s}\right) ^{H-\frac{1}{2}}(t-s)^{H-%
\frac{1}{2}}+\left( \frac{1}{2}-H\right) s^{\frac{1}{2}-H}\int_{s}^{t}u^{H-%
\frac{3}{2}}(u-s)^{H-\frac{1}{2}}du\right] .  \label{KH}
\end{equation}%
Here $c_{H}=\sqrt{\frac{2H}{(1-2H)\beta (1-2H,H+1/2)}}$ and $\beta $ is the
Beta function. See \cite[Proposition 5.1.3]{Nualart}.

Based on the kernel $K_{H}$, one can introduce by means (\ref{2.2}) an
isometry $K_{H}^{\ast }$ between $\mathcal{E}$ and $L^{2}([0,T])$ such that $%
(K_{H}^{\ast }1_{[0,t]})(s)=K_{H}(t,s)1_{[0,t]}(s).$ This isometry has an
extension to the Hilbert space $\mathcal{H}$, which has the following
representations by means of fractional derivatives

\begin{equation*}
(K_{H}^{\ast }\varphi )(s)=c_{H}\Gamma \left( H+\frac{1}{2}\right) s^{\frac{1%
}{2}-H}\left( D_{T^{-}}^{\frac{1}{2}-H}u^{H-\frac{1}{2}}\varphi (u)\right)
(s)
\end{equation*}%
and 
\begin{align*}
(K_{H}^{\ast }\varphi )(s)=& \,c_{H}\Gamma \left( H+\frac{1}{2}\right)
\left( D_{T^{-}}^{\frac{1}{2}-H}\varphi (s)\right) (s) \\
& +c_{H}\left( \frac{1}{2}-H\right) \int_{s}^{T}\varphi (t)(t-s)^{H-\frac{3}{%
2}}\left( 1-\left( \frac{t}{s}\right) ^{H-\frac{1}{2}}\right) dt.
\end{align*}%
for $\varphi \in \mathcal{H}$. One also proves that $\mathcal{H}=I_{T^{-}}^{%
\frac{1}{2}-H}(L^{2})$. See \cite{DU} and \cite[Proposition 6]%
{alos.mazet.nualart.01}.

Since $K_{H}^{\ast }$ is an isometry from $\mathcal{H}$ into $L^{2}([0,T])$,
the $d$-dimensional process $W=\{W_{t},t\in \lbrack 0,T]\}$ defined by 
\begin{equation}
W_{t}:=B^{H}((K_{H}^{\ast })^{-1}(1_{[0,t]}))  \label{WBH}
\end{equation}%
is a Wiener process and the process $B^{H}$ can be represented as 
\begin{equation}
B_{t}^{H}=\int_{0}^{t}K_{H}(t,s)dW_{s}.  \label{BHW}
\end{equation}%
See \cite{alos.mazet.nualart.01}.

In what follows we also need the Definition of a fractional Brownian motion
with respect to a filtration.

\begin{definition}
Let $\mathcal{G}=\left\{ \mathcal{G}_{t}\right\} _{t\in \left[ 0,T\right] }$
be a filtration on $\left( \Omega ,\mathcal{F},P\right) $ satisfying the
usual conditions. A fractional Brownian motion $B^{H}$ is called a $\mathcal{%
G}$-fractional Brownian motion if the process $W$ defined by (\ref{WBH}) is
a $\mathcal{G}$-Brownian motion.
\end{definition}

\bigskip

In the following, let $W$ be a standard Wiener process on a filtered
probability space $(\Omega ,\mathfrak{A},P),\{\mathcal{F}_{t}\}_{t\in
\lbrack 0,T]},$ where $\mathcal{F}=\{\mathcal{F}_{t}\}_{t\in \lbrack 0,T]}$
is the natural filtration generated by $W$ and augmented by all $P$-null
sets. Denote by $B:=B^{H}$ the fractional Brownian motion with Hurst
parameter $H\in (0,1/2)$ as in (\ref{BHW}).

We aim at using a version of Girsanov's theorem for fractional Brownian
motion which is due to \cite[Theorem 4.9]{DU}. The version stated here
corresponds to that in \cite[Theorem 2]{NO}. To this end, we need the
definition of an isomorphism $K_{H}$ from $L^{2}([0,T])$ onto $I_{0+}^{H+%
\frac{1}{2}}(L^{2})$ with respec to the kernel $K_{H}(t,s)$ in terms of the
fractional integrals as follows (see \cite[Theorem 2.1]{DU}): 
\begin{equation*}
(K_{H}\varphi )(s)=I_{0^{+}}^{2H}s^{\frac{1}{2}-H}I_{0^{+}}^{\frac{1}{2}%
-H}s^{H-\frac{1}{2}}\varphi ,\quad \varphi \in L^{2}([0,T]).
\end{equation*}

Using this and the properties of the Riemann-Liouville fractional integrals
and derivatives, one can show that the inverse of $K_{H}$ can be represented
as 
\begin{equation*}
(K_{H}^{-1}\varphi )(s)=s^{\frac{1}{2}-H}D_{0^{+}}^{\frac{1}{2}-H}s^{H-\frac{%
1}{2}}D_{0^{+}}^{2H}\varphi (s),\quad \varphi \in I_{0+}^{H+\frac{1}{2}%
}(L^{2}).
\end{equation*}

From this one obtains for absolutely continuous functions $\varphi $ (see 
\cite{NO}) that 
\begin{equation*}
(K_{H}^{-1}\varphi )(s)=s^{H-\frac{1}{2}}I_{0^{+}}^{\frac{1}{2}-H}s^{\frac{1%
}{2}-H}\varphi ^{\prime }(s).
\end{equation*}

\begin{theorem}[Girsanov's theorem for fBm]
\label{girsanov} Let $u=\{u_{t},t\in \lbrack 0,T]\}$ be an $\mathcal{F}$%
-adapted process with integrable trajectories and set $\widetilde{B}%
_{t}^{H}=B_{t}^{H}+\int_{0}^{t}u_{s}ds,\quad t\in \lbrack 0,T].$ Suppose that

\begin{itemize}
\item[(i)] $\int_{0}^{\cdot }u_{s}ds\in I_{0+}^{H+\frac{1}{2}}(L^{2}([0,T]))$%
, $P$-a.s.

\item[(ii)] $E[\xi_T]=1$ where 
\begin{equation*}
\xi_T := \exp\left\{-\int_0^T K_H^{-1}\left( \int_0^{\cdot} u_r
dr\right)(s)dW_s - \frac{1}{2} \int_0^T K_H^{-1} \left( \int_0^{\cdot} u_r
dr \right)^2(s)ds \right\}.
\end{equation*}
\end{itemize}

Then the shifted process $\widetilde{B}^H$ is an $\mathcal{F}$-fractional
Brownian motion with Hurst parameter $H$ under the new probability $%
\widetilde{P}$ defined by $\frac{d\widetilde{P}}{dP}=\xi_T$.
\end{theorem}

\begin{remark}
In the the multi-dimensional case, we define 
\begin{equation*}
(K_{H}\varphi )(s):=((K_{H}\varphi ^{(1)})(s),\dots ,(K_{H}\varphi
^{(d)})(s))^{\ast },\quad \varphi \in L^{2}([0,T];\mathbb{R}^{d}),
\end{equation*}%
where $\ast $ denotes transposition. Similarly for $K_{H}^{-1}$ and $%
K_{H}^{\ast }$.
\end{remark}

In this article we also resort to the following technical lemma (see \cite[%
Lemma 4.3]{BNP} for a proof):

\begin{lemma}
\label{Novikov} Let $\tilde{B}_{t}^{H}$ be a $d$-dimensional fractional
Brownian motion with respect to $(\Omega ,\mathfrak{A},\tilde{P})$. Then for
every $k\in \mathbb{R}$ we have 
\begin{equation*}
\tilde{E}\left[ \exp \left\{ k\int_{0}^{T}\left\vert K_{H}^{-1}\left(
\int_{0}^{\cdot }b(r,\tilde{B}_{r}^{H})dr\right) (s)\right\vert
^{2}ds\right\} \right] \leq C_{H,d,\mu ,T}(\Vert b\Vert _{L_{\infty
}^{\infty }})
\end{equation*}%
for some continuous increasing function $C_{H,d,k,T}$ depending only on $H$, 
$d$, $T$ and $k$.

In particular, 
\begin{equation*}
\tilde{E}\left[ \mathcal{E}\left( \int_{0}^{T}K_{H}^{-1}\left(
\int_{0}^{\cdot }b(r,\tilde{B}_{r}^{H})dr\right) ^{\ast }(s)dW_{s}\right)
^{p}\right] \leq C_{H,d,\mu ,T}(\Vert b\Vert _{L_{\infty }^{\infty }}),
\end{equation*}%
where $\mathcal{E}(M_{t})$ is the Dolean-Dade exponential of a local
martingale $M_{t},0\leq t\leq T$ and where $\tilde{E}$ denotes expectation
under $\tilde{P}$ and $\ast $ transposition.
\end{lemma}

\bigskip

In this paper, we will also make use of an integration by parts formula for
iterated integrals based on \emph{shuffle permutations}. For this purpose,
let $m$ and $n$ be integers. Denote by $S(m,n)$ the set of shuffle
permutations, i.e. the set of permutations $\sigma :\{1,\dots
,m+n\}\rightarrow \{1,\dots ,m+n\}$ such that $\sigma (1)<\dots <\sigma (m)$
and $\sigma (m+1)<\dots <\sigma (m+n)$.

Introduce the $m$-dimensional simplex for $0\leq \theta <t\leq T$, 
\begin{equation*}
\Delta _{\theta ,t}^{m}:=\{(s_{m},\dots ,s_{1})\in \lbrack 0,T]^{m}:\,\theta
<s_{m}<\cdots <s_{1}<t\}.
\end{equation*}%
The product of two simplices can be represented as follows 
\begin{equation*}
\Delta _{\theta ,t}^{m}\times \Delta _{\theta ,t}^{n}=%
\mbox{\footnotesize
$\bigcup_{\sigma \in S(m,n)} \{(w_{m+n},\dots,w_1)\in [0,T]^{m+n} : \,
\theta< w_{\sigma(m+n)} <\cdots < w_{\sigma(1)} <t\} \cup \mathcal{N}$
\normalsize},
\end{equation*}%
where the set $\mathcal{N}$ has null Lebesgue measure. So, if $%
f_{i}:[0,T]\rightarrow \mathbb{R}$, $i=1,\dots ,m+n$ are integrable
functions we get that 
\begin{align}
\int_{\Delta _{\theta ,t}^{m}}\prod_{j=1}^{m}f_{j}(s_{j})ds_{m}\dots ds_{1}&
\int_{\Delta _{\theta ,t}^{n}}\prod_{j=m+1}^{m+n}f_{j}(s_{j})ds_{m+n}\dots
ds_{m+1}  \notag \\
& =\sum_{\sigma \in S(m,n)}\int_{\Delta _{\theta
,t}^{m+n}}\prod_{j=1}^{m+n}f_{\sigma (j)}(w_{j})dw_{m+n}\cdots dw_{1}.
\label{shuffleIntegral}
\end{align}

A generalization of the latter relation is the following (see \cite{BNP}):

\begin{lemma}
\label{partialshuffle} Let $n,$ $p$ and $k$ be non-negative integers, $k\leq
n$. Suppose we have integrable functions $f_{j}:[0,T]\rightarrow \mathbb{R}$%
, $j=1,\dots ,n$ and $g_{i}:[0,T]\rightarrow \mathbb{R}$, $i=1,\dots ,p$. We
may then write 
\begin{align*}
& \int_{\Delta _{\theta ,t}^{n}}f_{1}(s_{1})\dots f_{k}(s_{k})\int_{\Delta
_{\theta ,s_{k}}^{p}}g_{1}(r_{1})\dots g_{p}(r_{p})dr_{p}\dots
dr_{1}f_{k+1}(s_{k+1})\dots f_{n}(s_{n})ds_{n}\dots ds_{1} \\
& =\sum_{\sigma \in A_{n,p}}\int_{\Delta _{\theta ,t}^{n+p}}h_{1}^{\sigma
}(w_{1})\dots h_{n+p}^{\sigma }(w_{n+p})dw_{n+p}\dots dw_{1},
\end{align*}%
where $h_{l}^{\sigma }\in \{f_{j},g_{i}:1\leq j\leq n,1\leq i\leq p\}$.
Above $A_{n,p}$ stands for a subset of permutations of $\{1,\dots ,n+p\}$
such that $\#A_{n,p}\leq C^{n+p}$ for an appropriate constant $C\geq 1$.
Here $s_{0}:=\theta $.
\end{lemma}

\bigskip

The proof of Lemma \ref{SecondBoundedDerivative} relies on an important
estimate (see e.g. Proposition 3.3 in \cite{BLPP} for a new proof (2018)).
In order to state this result, we need some notation. Let $m$ be an integer
and let $f:[0,T]^{m}\times (\mathbb{R}^{d})^{m}\rightarrow \mathbb{R}$ be a
function of the form 
\begin{equation}
f(s,z)=\prod_{j=1}^{m}f_{j}(s_{j},z_{j}),\quad s=(s_{1},\dots ,s_{m})\in
\lbrack 0,T]^{m},\quad z=(z_{1},\dots ,z_{m})\in (\mathbb{R}^{d})^{m},
\label{ffnew}
\end{equation}%
where $f_{j}:[0,T]\times \mathbb{R}^{d}\rightarrow \mathbb{R}$, $j=1,\dots
,m $ are smooth functions with compact support. In addition, let $\varkappa
:[0,T]^{m}\rightarrow \mathbb{R}$ be a function of the form 
\begin{equation}
\varkappa (s)=\prod_{j=1}^{m}\varkappa _{j}(s_{j}),\quad s\in \lbrack
0,T]^{m},  \label{kappa}
\end{equation}%
where $\varkappa _{j}:[0,T]\rightarrow \mathbb{R}$, $j=1,\dots ,m$ are
integrable functions.

Further, denote by $\alpha _{j}$ a multi-index and $D^{\alpha _{j}}$ its
corresponding differential operator. For $\alpha =(\alpha _{1},\dots ,\alpha
_{m})$ as an element of $\mathbb{N}_{0}^{d\times m}$ with $|\alpha
|:=\sum_{j=1}^{m}\sum_{l=1}^{d}\alpha _{j}^{(l)}$, we write 
\begin{equation*}
D^{\alpha }f(s,z)=\prod_{j=1}^{m}D^{\alpha _{j}}f_{j}(s_{j},z_{j}).
\end{equation*}

\begin{theorem}
\label{mainestimate2} Let $B^{H},H\in (0,1/2)$ be a standard $d-$dimensional
fractional Brownian motion and functions $f$ and $\varkappa $ as in (\ref%
{ffnew}), respectively as in (\ref{kappa}). Let $\theta ,t\in \lbrack 0,T]$
with $\theta <t$ and%
\begin{equation*}
\varkappa _{j}(s)=(K_{H}(s,\theta ))^{\varepsilon _{j}},\theta <s<t
\end{equation*}%
for every $j=1,...,m$ with $(\varepsilon _{1},...,\varepsilon _{m})\in
\{0,1\}^{m}.$ Let $\alpha \in (\mathbb{N}_{0}^{d})^{m}$ be a multi-index. If 
\begin{equation*}
H<\frac{\frac{1}{2}-\gamma }{(d-1+2\sum_{l=1}^{d}\alpha _{j}^{(l)})}
\end{equation*}%
for all $j$, where $\gamma \in (0,H)$ is sufficiently small, then there
exists a universal constant $C$ (depending on $H$, $T$ and $d$, but
independent of $m$, $\{f_{i}\}_{i=1,...,m}$ and $\alpha $) such that for any 
$\theta ,t\in \lbrack 0,T]$ with $\theta <t$ we have%
\begin{eqnarray*}
&&\left\vert E\int_{\Delta _{\theta ,t}^{m}}\left( \prod_{j=1}^{m}D^{\alpha
_{j}}f_{j}(s_{j},B_{s_{j}}^{H})\varkappa _{j}(s_{j})\right) ds\right\vert \\
&\leq &C^{m+\left\vert \alpha \right\vert }\prod_{j=1}^{m}\left\Vert
f_{j}(\cdot ,z_{j})\right\Vert _{L^{1}(\mathbb{R}^{d};L^{\infty
}([0,T]))}\theta ^{(H-\frac{1}{2})\sum_{j=1}^{m}\varepsilon _{j}} \\
&&\times \frac{(\prod_{l=1}^{d}(2\left\vert \alpha ^{(l)}\right\vert
)!)^{1/4}(t-\theta )^{-H(md+2\left\vert \alpha \right\vert )-(H-\frac{1}{2}%
-\gamma )\sum_{j=1}^{m}\varepsilon _{j}+m}}{\Gamma (-H(2md+4\left\vert
\alpha \right\vert )+2(H-\frac{1}{2}-\gamma )\sum_{j=1}^{m}\varepsilon
_{j}+2m)^{1/2}}.
\end{eqnarray*}
\end{theorem}

\begin{remark}
\label{S}The above theorem remains valid for time-homogeneous functions $%
\{f_{i}\}_{i=1,...,m}$ in the Schwartz function space.
\end{remark}

\bigskip

The proof of Lemma \ref{SecondBoundedDerivative} also requires the following
auxiliary result:

\begin{lemma}
\label{OrderDerivatives}Let $n,$ $p$ and $k$ be non-negative integers, $%
k\leq n$. Assume we have functions $f_{j}:[0,T]\rightarrow \mathbb{R}$, $%
j=1,\dots ,n$ and $g_{i}:[0,T]\rightarrow \mathbb{R}$, $i=1,\dots ,p$ such
that 
\begin{equation*}
f_{j}\in \left\{ \frac{\partial ^{\alpha _{j}^{(1)}+...+\alpha _{j}^{(d)}}}{%
\partial ^{\alpha _{j}^{(1)}}x_{1}...\partial ^{\alpha _{j}^{(d)}}x_{d}}%
b^{(r)}(u,X_{u}^{x}),\text{ }r=1,...,d\right\} ,\text{ }j=1,...,n
\end{equation*}%
and 
\begin{equation*}
g_{i}\in \left\{ \frac{\partial ^{\beta _{i}^{(1)}+...+\beta _{i}^{(d)}}}{%
\partial ^{\beta _{i}^{(1)}}x_{1}...\partial ^{\beta _{i}^{(d)}}x_{d}}%
b^{(r)}(u,X_{u}^{x}),\text{ }r=1,...,d\right\} ,\text{ }i=1,...,p
\end{equation*}%
for $\alpha :=(\alpha _{j}^{(l)})\in \mathbb{N}_{0}^{d\times n}$ and $\beta
:=(\beta _{i}^{(l)})\in \mathbb{N}_{0}^{d\times p},$ where $X_{\cdot }^{x}$
is the strong solution to 
\begin{equation*}
X_{t}^{x}=x+\int_{0}^{t}b(u,X_{u}^{x})du+B_{t}^{H},\text{ }0\leq t\leq T
\end{equation*}%
for $b=(b^{(1)},...,b^{(d)})$ with $b^{(r)}\in C_{c}((0,T)\times \mathbb{R}%
^{d})$ for all $r=1,...,d$. So (as we shall say in the sequel) the product $%
g_{1}(r_{1})\cdot \dots \cdot g_{p}(r_{p})$ has a total order of derivatives 
$\left\vert \beta \right\vert =\sum_{l=1}^{d}\sum_{i=1}^{p}\beta _{i}^{(l)}$%
. We know from Lemma \ref{partialshuffle} that 
\begin{align}
& \int_{\Delta _{\theta ,t}^{n}}f_{1}(s_{1})\dots f_{k}(s_{k})\int_{\Delta
_{\theta ,s_{k}}^{p}}g_{1}(r_{1})\dots g_{p}(r_{p})dr_{p}\dots
dr_{1}f_{k+1}(s_{k+1})\dots f_{n}(s_{n})ds_{n}\dots ds_{1}  \notag \\
& =\sum_{\sigma \in A_{n,p}}\int_{\Delta _{\theta ,t}^{n+p}}h_{1}^{\sigma
}(w_{1})\dots h_{n+p}^{\sigma }(w_{n+p})dw_{n+p}\dots dw_{1},  \label{h}
\end{align}%
where $h_{l}^{\sigma }\in \{f_{j},g_{i}:1\leq j\leq n,$ $1\leq i\leq p\}$, $%
A_{n,p}$ is a subset of permutations of $\{1,\dots ,n+p\}$ such that $%
\#A_{n,p}\leq C^{n+p}$ for an appropriate constant $C\geq 1$, and $%
s_{0}=\theta $. Then the products%
\begin{equation*}
h_{1}^{\sigma }(w_{1})\cdot \dots \cdot h_{n+p}^{\sigma }(w_{n+p})
\end{equation*}%
have a total order of derivatives given by $\left\vert \alpha \right\vert
+\left\vert \beta \right\vert .$
\end{lemma}

\begin{proof}
The result is proved by induction on $n$. For $n=1$ and $k=0$ the result is
trivial. For $k=1$ we have 
\begin{eqnarray*}
\int_{\theta }^{t}f_{1}(s_{1})\int_{\Delta _{\theta
,s_{1}}^{p}}g_{1}(r_{1})\dots g_{p}(r_{p}) &&dr_{p}\dots dr_{1}ds_{1} \\
&=&\int_{\Delta _{\theta ,t}^{p+1}}f_{1}(w_{1})g_{1}(w_{2})\dots
g_{p}(w_{p+1})dw_{p+1}\dots dw_{1},
\end{eqnarray*}%
where we have put $w_{1}=s_{1},$ $w_{2}=r_{1},\dots ,w_{p+1}=r_{p}$. Hence
the total order of derivatives involved in the product of the last integral
is given by $\sum_{l=1}^{d}\alpha
_{1}^{(l)}+\sum_{l=1}^{d}\sum_{i=1}^{p}\beta _{i}^{(l)}=\left\vert \alpha
\right\vert +\left\vert \beta \right\vert .$

Assume the result holds for $n$ and let us show that this implies that the
result is true for $n+1$. Either $k=0,1$ or $2\leq k\leq n+1$. For $k=0$ the
result is trivial. For $k=1$ we have 
\begin{align*}
\int_{\Delta _{\theta ,t}^{n+1}}& f_{1}(s_{1})\int_{\Delta _{\theta
,s_{1}}^{p}}g_{1}(r_{1})\dots g_{p}(r_{p})dr_{p}\dots
dr_{1}f_{2}(s_{2})\dots f_{n+1}(s_{n+1})ds_{n+1}\dots ds_{1} \\
& =\int_{\theta }^{t}f_{1}(s_{1})\left( \int_{\Delta _{\theta
,s_{1}}^{n}}\int_{\Delta _{\theta ,s_{1}}^{p}}g_{1}(r_{1})\dots
g_{p}(r_{p})dr_{p}\dots dr_{1}f_{2}(s_{2})\dots
f_{n+1}(s_{n+1})ds_{n+1}\dots ds_{2}\right) ds_{1}.
\end{align*}%
From (\ref{shuffleIntegral}) we observe by using the shuffle permutations
that the latter inner double integral on diagonals can be written as a sum
of integrals on diagonals of length $p+n$ with products having a total order
of derivatives given by $\sum_{l=1}\sum_{j=2}^{n+1}\alpha
_{j}^{(l)}+\sum_{l=1}^{d}\sum_{i=1}^{p}\beta _{i}^{(l)}$. Hence we obtain a
sum of products, whose total order of derivatives is $\sum_{l=1}^{d}%
\sum_{j=2}^{n+1}\alpha _{j}^{(l)}+\sum_{l=1}^{d}\sum_{i=1}^{p}\beta
_{i}^{(l)}+\sum_{l=1}^{d}\alpha _{1}^{(l)}=\left\vert \alpha \right\vert
+\left\vert \beta \right\vert .$

For $k\geq 2$ we have (in connection with Lemma \ref{partialshuffle}) from
the induction hypothesis that 
\begin{align*}
\int_{\Delta _{\theta ,t}^{n+1}}f_{1}(s_{1})\dots f_{k}(s_{k})\int_{\Delta
_{\theta ,s_{k}}^{p}}g_{1}(r_{1})\dots g_{p}(r_{p})& dr_{p}\dots
dr_{1}f_{k+1}(s_{k+1})\dots f_{n+1}(s_{n+1})ds_{n+1}\dots ds_{1} \\
=\int_{\theta }^{t}f_{1}(s_{1})\int_{\Delta _{\theta
,s_{1}}^{n}}f_{2}(s_{2})\dots f_{k}(s_{k})& \int_{\Delta _{\theta
,s_{k}}^{p}}g_{1}(r_{1})\dots g_{p}(r_{p})dr_{p}\dots dr_{1} \\
& \times f_{k+1}(s_{k+1})\dots f_{n+1}(s_{n+1})ds_{n+1}\dots ds_{2}ds_{1} \\
=\sum_{\sigma \in A_{n,p}}\int_{\theta }^{t}f_{1}(s_{1})\int_{\Delta
_{\theta ,s_{1}}^{n+p}}& h_{1}^{\sigma }(w_{1})\dots h_{n+p}^{\sigma
}(w_{n+p})dw_{n+p}\dots dw_{1}ds_{1},
\end{align*}%
where each of the products $h_{1}^{\sigma }(w_{1})\cdot \dots \cdot
h_{n+p}^{\sigma }(w_{n+p})$ have a total order of derivatives given by $%
\sum_{l=1}\sum_{j=2}^{n+1}\alpha
_{j}^{(l)}+\sum_{l=1}^{d}\sum_{i=1}^{p}\beta _{i}^{(l)}.$ Thus we get a sum
with respect to a set of permutations $A_{n+1,p}$ with products having a
total order of derivatives which is%
\begin{equation*}
\sum_{l=1}^{d}\sum_{j=2}^{n+1}\alpha
_{j}^{(l)}+\sum_{l=1}^{d}\sum_{i=1}^{p}\beta _{i}^{(l)}+\sum_{l=1}^{d}\alpha
_{1}^{(l)}=\left\vert \alpha \right\vert +\left\vert \beta \right\vert .
\end{equation*}
\end{proof}

\end{document}